\documentclass[11pt,a4paper]{amsart}
\usepackage{amsmath,amssymb,color,multicol,setspace}
\usepackage[utf8,utf8x]{inputenc}

\usepackage{fullpage}

\usepackage[pdftitle = {Frieze\ patterns\ with\ coefficients}]{hyperref}
\usepackage{longtable}
\usepackage{xy,amscd}
\usepackage{comment}
\usepackage{epsfig}
\usepackage{rotating}
\usepackage{xspace}
\usepackage{enumitem}
\usepackage{float}
\xyoption{all}
\usepackage{tikz}
\usetikzlibrary{arrows,decorations.pathmorphing,decorations.pathreplacing,positioning,shapes.geometric,shapes.misc,decorations.markings,decorations.fractals,calc,patterns}

\newcommand{\blue}{\color{blue}}
\newcommand{\red}{\color{red}}
\newcommand{\green}{\color{green}}

\newtheorem{Lemma}{Lemma}[section]
\newtheorem{Theorem}[Lemma]{Theorem}
\newtheorem{Proposition}[Lemma]{Proposition}
\newtheorem{Corollary}[Lemma]{Corollary}

\theoremstyle{definition}
\newtheorem{Definition}[Lemma]{Definition}

\newtheorem{Remark}[Lemma]{Remark}

\newtheorem{Example}[Lemma]{Example}

\numberwithin{equation}{section}

\newcommand{\NN }{\mathbb{N}}
\newcommand{\CC }{\mathbb{C}}

\newcommand{\ZZ }{\mathbb{Z}}

\newcommand{\id }{\mathrm{id}}

\DeclareMathOperator{\SL}{SL}

\DeclareMathOperator{\md}{mod}

\newcommand{\ta }{\eta}
\newcommand{\emu }{\mu}

\title[Frieze patterns with coefficients]
{Frieze patterns with coefficients}

\author{Michael~Cuntz}
\address{Michael Cuntz, Leibniz Universit\"at Hannover,
Institut f\"ur Algebra, Zahlentheorie und Diskrete Mathematik,
Fakult\"at f\"ur Mathematik und Physik,
Welfengarten 1,
D-30167 Hannover, Germany}
\email{cuntz@math.uni-hannover.de}
\urladdr{https://www.iazd.uni-hannover.de/cuntz.html}

\author{Thorsten~Holm}
\address{Thorsten Holm, Leibniz Universit\"at Hannover,
Institut f\"ur Algebra, Zahlentheorie und Diskrete Mathematik,
Fakult\"at f\"ur Mathematik und Physik,
Welfengarten 1,
D-30167 Hannover, Germany}
\email{holm@math.uni-hannover.de}
\urladdr{http://www2.iazd.uni-hannover.de/\~{ }tholm}

\author{Peter~J{\o}rgensen}
\address{Peter J{\o}rgensen, School of Mathematics and Statistics,
Newcastle University, Newcastle upon Tyne NE1 7RU,
United Kingdom}
\email{peter.jorgensen@ncl.ac.uk}
\urladdr{https://www.staff.ncl.ac.uk/peter.jorgensen/}



\begin{document}

\begin{abstract}
Frieze patterns, as introduced by Coxeter in the 1970's, are closely related to cluster algebras without coefficients. A suitable generalization of frieze patterns, linked to cluster algebras with coefficients, has only briefly appeared in an unpublished manuscript by Propp. In this paper we study these frieze patterns with coefficients systematically and prove various fundamental results, generalizing classic results for frieze patterns. As a consequence we see how frieze patterns with coefficients can be obtained from classic frieze patterns by cutting out subpolygons from the triangulated polygons associated to classic Conway-Coxeter frieze patterns. We address the question of which frieze patterns with coefficients can be obtained in this way and solve this problem completely for triangles. Finally, we prove a finiteness result for frieze patterns with coefficients by showing that for a given boundary sequence there are only finitely many (non-zero) frieze patterns with coefficients with entries in a discrete subset of the complex numbers.
\end{abstract}

\maketitle


\section{Introduction}

Frieze patterns have been introduced by Coxeter
\cite{Cox71} and a beautiful theory for frieze patterns over positive integers 
has been developed subsequently by Conway and Coxeter \cite{CC73}. Only
some three decades later the importance of frieze patterns for other 
areas of mathematics has become clearer when Fomin and Zelevinsky invented 
cluster algebras \cite{FZI}. Namely, the exchange condition in cluster algebras 
mimics the local condition defining frieze patterns. In this way, 
starting with a set of indeterminates, the entries in a frieze pattern
(over the field of rational functions) 
are precisely the cluster variables of the corresponding cluster algebra of Dynkin type $A$ with trivial coefficients.
Via cluster algebras, frieze patterns are now connected to many areas of mathematics
and therefore, frieze patterns are currently a very active area of research;
see the survey by Morier-Genoud \cite{MG15} for more details. 

Classically, a frieze pattern is an infinite array of the form
\[
\begin{array}{ccccccccccc}
 & & \ddots & & & &\ddots  & & & \\
~0~ & ~1~ & ~c_{i-1,i+1}~ & ~c_{i-1,i+2}~ & ~\cdots~ & ~\cdots~ & ~c_{i-1,n+i}~ & ~1~ & ~0~ & & \\
& 0 & 1 & c_{i,i+2} & c_{i,i+3} & \cdots & \cdots & c_{i,n+i+1} & ~1~ & ~0~ & \\
& & 0 & 1 & c_{i+1,i+3} & c_{i+1,i+4} & \cdots & \cdots & c_{i+1,n+i+2} & ~1~ & ~0~ \\
 & & & & \ddots  & & & &\ddots  & 
\end{array}
\]
such that all neighbouring $2\times 2$-matrices have determinant 1.

From the cluster algebra
viewpoint the two bounding diagonals of $1$'s mean to set all the 
coefficients (or frozen variables) in the
cluster algebra equal to 1. Basically, the entire extensive recent 
literature on frieze patterns deals with such classic frieze patterns. 

In the theory of cluster algebras the coefficients
are very important, so it would be natural to consider more general frieze 
patterns where the bounding diagonals can have arbitrary entries. Such frieze
patterns have appeared briefly already in an unpublished manuscript by Propp
\cite {Propp}. In particular, one finds there the modified local condition
on $2\times 2$-determinants, now involving the boundary entries. This 
is set up in such a way that, again, the entries in a certain frieze pattern
(over the field of rational functions) are precisely the cluster variables
of the corresponding cluster algebra with coefficients.

Surprisingly, it seems that this topic has so far not been 
considered further in the literature. 
One aim of the present paper is to develop the general theory of these
more general frieze patterns, which we call {\em frieze patterns with
coefficients}; see Definition \ref{def:frieze} for a precise 
definition. As in the classic case, a satisfactory theory can only be expected
for tame frieze patterns with coefficients, i.e.\ imposing that all adjacent
$3\times 3$-determinants are 0. (There are far too many wild frieze 
patterns, see \cite{Cuntz-wild}.)

In Sections \ref{sec:definition} and \ref{sec:ptolemy} we develop
a theory for tame frieze patterns with coefficients, generalizing several of
the well-known results for classic frieze patterns. 

We introduce certain $2\times 2$-matrices, called $\emu$-matrices, which
govern the propagation along the rows and the columns of a frieze pattern with
coefficients, that is, multiplication with such matrices transforms two
consecutive entries in a row or column to the next two consecutive entries.
See Definition \ref{def:etamatrix_rev} and Proposition \ref{propagation_rev}. 
Moreover, we present in Corollary \ref{cor:cij} formulae for
how every entry in the frieze pattern appears in 
a certain product of $\emu$-matrices. 

As another application of the propagation formulae we prove that 
every tame frieze pattern with coefficients satisfies a 
glide reflection, see Theorem \ref{thm:glide}. A fundamental domain
for the entries of the tame frieze pattern with coefficients under the
glide reflection can be indexed
in such a way that it corresponds bijectively to the edges and diagonals
of a regular $(n+3)$-gon (where $n$ is the height of the frieze pattern). 

This is the starting point in Section \ref{sec:ptolemy}. It
allows a useful alternative viewpoint for tame frieze patterns with
coefficients: the entries in such a frieze pattern become the labels for
the edges and diagonals of a polygon (where the boundary entries become 
the labels for the edges). We show in Theorem \ref{thm:ptolemy}
that these labels have to satisfy many more 
than the defining local 
conditions, namely the so-called Ptolemy relations. 
 
For distinguishing the two viewpoints we speak of a tame {\em frieze with 
coefficients} if we consider a map from the edges and 
diagonals of a polygon satisfying all Ptolemy relations.

In Section \ref{sec:subpolygons} we present a construction for 
obtaining friezes with coefficients from classic friezes: given a classic 
frieze one can cut out any subpolygon and get a frieze with coefficients.
If this frieze with coefficients is a classic frieze, then it is a summand in the sense of \cite{Cuntz-comb} or \cite{HJ-pang}.

This naturally leads to the question which friezes with coefficients can
be obtained from classic friezes by cutting out subpolygons. 
Small examples already show that not every frieze with coefficients can be 
obtained from classic friezes in this way.

The rest of Section \ref{sec:subpolygons} and also Section \ref{sec:triangles} 
deal with this fundamental question for friezes with coefficients over positive
integers. 
Our results not only give insight into friezes with coefficients
but also shed new light on the classic Conway-Coxeter friezes, i.e. friezes
over positive integers with all boundary entries equal to 1. 

In Lemma \ref{lem:gcd} we show that any triangle cut out of a classic Conway-Coxeter
frieze has the property that the greatest common divisor of any two of the labels
on the triangle divides the third label. 
Conversely, we show in Lemma \ref{accordion} that given two coprime natural
numbers $a,b$ there exists a classic Conway-Coxeter frieze containing a triangle 
with labels $1,a,b$. The proof of this result exhibits a close connection between
friezes and the Euclidean Algorithm, which we think might be of interest in its
own right. 

Note that the results just mentioned do not completely settle the question
which labelled triangles can actually appear as subpolygons of classic
Conway-Coxeter friezes. It turns out that this problem is rather subtle
and we address it in Section \ref{sec:triangles}. As a main result we give 
a complete classification of the problem, that is, we describe the 
triples $(a,b,c)\in \mathbb{N}^3$ such that there exists a triangle
in some classic Conway-Coxeter frieze with labels $a,b,c$:

\begin{Theorem}[Thm.\ \ref{Delta}]
Let $a,b,c\in \NN$.
Then the triple $(a,b,c)$ appears as labels of a triangle in a classic Conway-Coxeter frieze if and only if
$\gcd(a,b)=\gcd(b,c)=\gcd(a,c)$ and
\[ \nu_2(a)=\nu_2(b)=\nu_2(c)=0\quad  \text{or} \quad
|\{\nu_2(a),\nu_2(b),\nu_2(c)\}|>1, \]
where $\nu_2(n)$ is the $2$-valuation of $n$.
\end{Theorem}

In Section 6 we go back to frieze patterns with coefficients having entries in any subset 
$R\subseteq \mathbb{C}$ (not necessarily integers as in the previous section). We
address another fundamental question: for a given boundary sequence, how many 
frieze patterns with
coefficients exist over $R$? Our main interest is whether there are finitely
or infinitely many such frieze patterns. Easy examples show that allowing entries to
be 0 rather quickly leads to infinitely many frieze patterns. So we restrict to friezes
patterns
with non-zero entries. 

In Lemma 6.1 we prove a rather general finiteness result, which for a given
boundary sequence yields an upper bound for the values in the quiddity cycle
of the corresponding frieze pattern (i.e. for the entries in the diagonal 
next to the boundary diagonal, or the labels of the length 2 diagonals in the
frieze). This generalizes a similar result for classic frieze patterns
from \cite[Theorem 3.6]{CH}.
As a consequence we can conclude that for every discrete subset $R$ of the complex
numbers and every boundary sequence there exist only finitely many frieze 
patterns over $R\setminus \{0\}$ with the given boundary sequence;
see Proposition \ref{cor:finite}.

\section{Definition and fundamental properties}
\label{sec:definition}

In this section we introduce frieze patterns with coefficients. This concept goes back
to an unpublished manuscript by Propp \cite{Propp}. A large amount of research has been 
carried out on classic frieze patterns, i.e. where all coefficients are set to 1.
A systematic treatment of the far more general frieze patterns with coefficients 
does not seem to be contained in the literature so far. Therefore, we give a detailed 
account here, in particular proving several fundamental properties of such frieze
patterns with coefficients, generalizing well-known properties of the classic 
frieze patterns.

\begin{Definition} \label{def:frieze}
Let $R\subseteq\CC$ be a subset of the complex numbers.
Let $n\in \mathbb{Z}_{\ge 0}$.

A \emph{frieze pattern with coefficients of height $n$} over $R$ 
is an infinite array of the form
\[
\begin{array}{ccccccccccc}
 & & \ddots & & & &\ddots  & & & \\
0 & c_{i-1,i} & c_{i-1,i+1} & c_{i-1,i+2} & \cdots & \cdots & c_{i-1,n+i} & c_{i-1,n+i+1} & 0 & & \\
& 0 & c_{i,i+1} & c_{i,i+2} & c_{i,i+3} & \cdots & \cdots & c_{i,n+i+1} & c_{i,n+i+2} & 0 & \\
& & 0 & c_{i+1,i+2} & c_{i+1,i+3} & c_{i+1,i+4} & \cdots & \cdots & c_{i+1,n+i+2} & c_{i+1,n+i+3} & 0 \\
 & & & & \ddots  & & & &\ddots  & 
\end{array}
\]
where we also set $c_{i,i}=0=c_{i,n+i+3}$ for all $i\in \mathbb{Z}$,
such that the following holds:
\begin{enumerate}
    \item[{(i)}] $c_{i,j}\in R$ for all $i\in\mathbb{Z}$ and $i< j< n+i+3$.
    \item[{(ii)}] $c_{i,i+1}\neq 0$ for all $i\in \mathbb{Z}$.
    \item[{(iii)}] For every
(complete) adjacent $2\times 2$-submatrix $\begin{pmatrix} c_{i,j} & c_{i,j+1}\\
c_{i+1,j} & c_{i+1,j+1} \end{pmatrix}$ we have
\begin{equation*}\tag{$E_{i,j}$}\label{eq:local}
c_{i,j} c_{i+1,j+1} - c_{i,j+1}c_{i+1,j} = c_{i+1,n+i+3}c_{j,j+1}.
\end{equation*}
\end{enumerate}
\end{Definition}

\begin{Remark} \label{rem:frieze}
\begin{enumerate}
    \item The diagonals with entries $c_{i,i+1}$ and $c_{i,n+i+2}$
    are called the {\em boundary} of the frieze pattern with coefficients. 
    Part (iii) of the definition says that
    the determinant of every (complete) adjacent $2\times 2$-submatrix
is given by the product of two specific entries on the boundary of the
frieze pattern, see the blue numbers in Figure \ref{fig:local}. 
\item 
Note that for $j=i+1$,
Equation ($E_{i,i+1}$) imposes that 
$$c_{i,i+1}c_{i+1,i+2} = c_{i+1,n+i+3}c_{i+1,i+2}.
$$
Since $c_{i+1,i+2}\neq 0$ by part (ii) and 
$R\subseteq \mathbb{C}$ has no zero divisors we conclude that 
\begin{equation} \label{eq:glide}
    c_{i,i+1} = c_{i+1,n+i+3} \mbox{\hskip0.3cm for all $i\in \mathbb{Z}$}. 
\end{equation}
This means that there is a glide symmetry on the boundary entries. 
In particular,
applying (\ref{eq:glide}) twice we get $c_{i,i+1}=c_{n+i+3,n+i+4}$ for all $i\in \mathbb{Z}$.
That is, 
for a frieze pattern with coefficients of height $n$ the boundary is periodic of period
$n+3$. Hence, the entire boundary is determined by a sequence of $n+3$ consecutive
boundary entries $(c_{i,i+1},\ldots, c_{n+i+2,n+i+3})$. 
\item
From the glide symmetry in (\ref{eq:glide}) it also follows that 
the local condition (\ref{eq:local}) in Definition \ref{def:frieze}
could as well be stated in terms of two other boundary entries,
namely the red numbers instead of the blue numbers in Figure \ref{fig:local}.
\end{enumerate}
\end{Remark}

\begin{Remark} \label{rem:scale}
Frieze patterns with coefficients can be scaled, that is, if 
$\mathcal{C}=(c_{i,j})$ is a frieze pattern with coefficients and 
$z\in \mathbb{C}\setminus \{0\}$ then $z\mathcal{C}:=(zc_{i,j})$ is again
a frieze pattern with coefficients (possibly over some other subset
$R\subseteq \mathbb{C}$). This follows immediately from 
Definition \ref{def:frieze}. 
\end{Remark}

\begin{figure}
    \centering
   $ \begin{array}{cccccccccccc}
   &  &  & &  & ~0~ &  & & &  & & \\
   &  &  & &  & ~{\red c_{-n-2+j,j}}~ &  & & &  & & \\
   &  &  & &  & ~\vdots~ &  & & &  & & \\
    &  &  & &  & ~\vdots~ &  & &  &  & & \\
    ~0~ & ~{\red c_{i,i+1}}~ & \ldots & & \ldots & ~c_{i,j}~ & ~c_{i,j+1}~ &  & &  & & \\
    &  &  &  & & ~c_{i+1,j}~ & ~c_{i+1,j+1}~ & \ldots & & \ldots & ~{\blue c_{i+1,n+i+3}}~ & ~0~ \\
    & & & & &  & ~\vdots~ &  & & & &  \\
   & & & & &  & ~\vdots~ &  & & & &  \\ 
   & & & & &  & ~{\blue c_{j,j+1}}~ &  & & & &  \\
   & & & & &  & ~0~ &  & & & &
    \end{array}
    $
    \caption{The local condition in a frieze pattern with coefficients.}
    \label{fig:local}
\end{figure}

\begin{Example} \label{ex:frieze}
\begin{enumerate}
\item Classic frieze patterns, as introduced by Coxeter \cite{Cox71}, 
are those frieze patterns with coefficients where all boundary entries are
equal to 1.
    \item All frieze patterns with coefficients of height 0 are of the form
    $$\begin{array}{ccccccccc}
     & \ddots &  & & & & & & \\
    0 & a & ~c~ & 0 & & & & & \\
    & 0 & ~b~ & ~a~ & 0 & & & & \\
    & & 0 & ~c~ & ~b~ & 0 & & & \\
    & & & 0 & ~a~ & ~c~ & 0 & & \\
    & & & & 0 & ~b~ & ~a~ & 0 & \\
    & & & & & 0 & ~c~ & ~b~ & 0 \\
    & & & & &  &  & & \ddots
    \end{array}
$$
with arbitrary non-zero numbers $a,b,c$. 
\item We describe all frieze patterns with coefficents of height 1 over $\mathbb{N}$. Take any four numbers $a,b,c,d\in \mathbb{N}$ and consider
the infinite array of the form (repeated periodically) 
$$\begin{array}{cccccccc}
 & \ddots & & & & & & \\
0 & ~a~ & ~x~ & ~d~ & 0 & & & \\
& 0 & ~b~ & ~y~ & ~a~ & 0 & & \\
& & 0 & ~c~ & ~z~ & ~b~ & 0 & \\
& & & 0 & ~d~ & ~w~ & ~c~ & 0 \\
& & &  &  &  & \ddots &
\end{array}
$$
It is easy to check from the local conditions that $x=z$ and $y=w$.
Then this is a frieze pattern with coefficients if and only if 
$xy = ac+bd$. In other words, given $a,b,c,d\in \mathbb{N}$ we obtain
for each divisor of $ac+bd$ such a frieze pattern with coefficients.

As an explicit example we have the following frieze 
pattern with coefficients over $\mathbb{N}$:
$$\begin{array}{cccccccc}
 & \ddots & & & & & & \\
0 & ~3~ & ~4~ & ~3~ & 0 & & & \\
& 0 & ~7~ & ~9~ & ~3~ & 0 & & \\
& & 0 & ~5~ & ~4~ & ~7~ & 0 & \\
& & & 0 & ~3~ & ~9~ & ~5~ & 0 \\
& & &  &  &  & \ddots &
\end{array}
$$
\end{enumerate}
\end{Example}

For general subsets $R\subseteq \mathbb{C}$ there are many frieze patterns
with coefficients having entries in $R$, and many of them do not exhibit nice 
symmetry properties. This is already the case for classic frieze patterns, see 
\cite{Cuntz-wild}. As in the classic case, a satisfactory theory can only be 
expected for tame frieze patterns.

\begin{Definition} \label{def:tame}
Let $\mathcal{C}$ be a frieze pattern with coefficients as in Definition \ref{def:frieze}.
Then $\mathcal{C}$ is called {\em tame} if every complete adjacent $3\times 3$-submatrix of 
$\mathcal{C}$ has determinant 0.
\end{Definition}

We aim at showing that every tame frieze pattern with coefficients
satisfies a glide symmetry. We have already seen in Remark \ref{rem:frieze}
that the boundary entries have a glide symmetry, namely we have
$c_{i,i+1} = c_{i+1,n+i+3}$ for all $i\in \mathbb{Z}$. 

For extending this to all entries of the frieze patterns with coefficients we need formulae
describing how entries are 'propagated' along rows and columns, that is, how 
to obtain two consecutive entries in a row or column from the previous two entries. 
To this end it is useful to slightly extend the definition of a frieze
pattern with coefficients by introducing two extra diagonals. 

\begin{Remark} \label{def:extfrieze}
Let $\mathcal{C}$ be a frieze pattern with coefficients over $R\subseteq\mathbb{C}$ as in
Definition \ref{def:frieze}. The corresponding extended frieze pattern with coefficients
$\hat{\mathcal{C}}$ is the infinite array of the form
{\footnotesize
\[
\begin{array}{cccccccccccc}
& & & \ddots & & & &\ddots  & & & & \\
-c_{i-2,i-1} & 0 & c_{i-1,i} & c_{i-1,i+1} & \cdots & \cdots & c_{i-1,n+i} & c_{i-1,n+i+1} & 0 & -c_{i-1,i} & & \\
& -c_{i-1,i} & 0 & c_{i,i+1} & c_{i,i+2} & \cdots & \cdots & c_{i,n+i+1} & c_{i,n+i+2} & 0 & -c_{i,i+1} & \\
& & -c_{i,i+1} & 0 & c_{i+1,i+2} & c_{i+1,i+3} & \cdots & \cdots & c_{i+1,n+i+2} & c_{i+1,n+i+3} & 0 & -c_{i+1,i+2} \\
& & & & & \ddots  & & & &\ddots  & & 
\end{array}
\]
}
That is, in addition to Definition \ref{def:frieze} we set 
$$c_{i,i-1}:= - c_{i-1,i} \mbox{\,\,\,\,\,and\,\,\,\,\,} c_{i,n+i+4}:=-c_{i,i+1} \mbox{\,\,\,\,for all $i\in \mathbb{Z}$}.
$$
We mention some fundamental properties of these extended frieze patterns with coefficients.
\begin{enumerate}
    \item All local conditions ($E_{i,j}$) also hold in $\hat{\mathcal{C}}$. In fact, 
    the only new equations to check are ($E_{i,i}$) and ($E_{i,n+i+3}$) for $i\in\mathbb{Z}$.
    The former one reads
    $$c_{i,i}c_{i+1,i+1} - c_{i,i+1} c_{i+1,i} = c_{i+1,n+i+3} c_{i,i+1}
    $$
    which is true since $c_{i,i}=0$ by Definition \ref{def:frieze}, $c_{i+1,i}=-c_{i,i+1}$
    by the above definition and $c_{i+1,n+i+3} = c_{i,i+1}$ by (\ref{eq:glide}). 
    Similarly, one checks that ($E_{i,n+i+3}$) holds.
    \item $\mathcal{C}$ is tame if and only if $\hat{\mathcal{C}}$ is tame (cf.\ Definition \ref{def:tame}).
    In fact, the new complete adjacent $3\times 3$-submatrices in $\hat{\mathcal{C}}$ are of the form
    $$\begin{pmatrix} c_{i-1,i} & c_{i-1,i+1} & c_{i-1,i+2} \\
    0 & c_{i,i+1} & c_{i,i+2} \\
    -c_{i,i+1} & 0 & c_{i+1,i+2} \end{pmatrix}
    \mbox{\,\,\,\,and\,\,\,\,} 
    \begin{pmatrix} c_{i-1,n+i+1} & 0 & -c_{i-1,i} \\
    c_{i,n+i+1} & c_{i,n+i+2} & 0 \\
    c_{i+1,n+i+1} & c_{i+1,n+i+2} & c_{i+1,n+i+3} \end{pmatrix}
    $$
    for $i\in \mathbb{Z}$.
    The matrix on the left can be shown to have determinant 0 by expanding along the first 
    column, using equation ($E_{i-1,i+1}$) and the glide symmetry formula (\ref{eq:glide}).
    Similarly, the matrix on the right has determinant 0. 
    \end{enumerate}
Instead of only adding two further diagonals to a frieze pattern with coefficients $\mathcal{C}$
one could extend $\mathcal{C}$ to an entire $SL_2$-tiling of the plane, still satisfying 
all local conditions and tameness. This is well-known for classic frieze patterns and can 
easily be transferred to frieze patterns with coefficients.
We do not introduce this viewpoint here since it is not used in the present paper. 
\end{Remark}

We can now start to develop the propagation formulae. 
It will turn out that the following matrices are crucial for this. 

\begin{Definition}[$\emu$-matrices and $\ta$-matrices] \label{def:etamatrix_rev}
For $c,d\in \CC$ and $e\in \CC\setminus \{0\}$, let
\begin{eqnarray}
\emu(c,d,e) &:=& \begin{pmatrix} 0 & -\frac{d}{e} \\ 1 & \frac{c}{e} \end{pmatrix}, \\
\ta(c,d,e) &:=& \begin{pmatrix} \frac{c}{e} & -\frac{d}{e} \\ 1 & 0 \end{pmatrix}.
\end{eqnarray}
\end{Definition}

\begin{Remark}
For classic frieze patterns (i.e.\ all coefficients are $1$), the matrices $\eta(c,1,1)$ are used throughout the literature.
For our purposes the matrices $\mu(c,d,e)$ are more convenient. They are closely linked, namely
\[ \emu(c,d,e) = \tau \ta(c,d,e)^T \tau, \]
if $\tau:=\begin{pmatrix} 0&1\\1&0 \end{pmatrix}$.
Moreover, if $d\ne 0$, then
\[ \ta(c,d,e)^{-1} = \tau \ta(c,e,d) \tau. \]
\end{Remark}

The following result provides the propagation formulae. Note that we 
write the entries as they appear in the frieze pattern with coefficients in Definition \ref{def:frieze},
that is as row vectors for the propagation along rows and as column 
vectors for the propagation along columns.

\begin{Proposition}\label{propagation_rev}
Let $R\subseteq \mathbb{C}$ be a subset.
Let $\mathcal{C}=(c_{i,j})$ be a tame frieze pattern with coefficients
over $R$ of height $n$, and $\hat{\mathcal{C}}$ the corresponding extended frieze pattern (see
Remark \ref{def:extfrieze}). We write $d_i:=c_{i,i+1}$, $c_i:=c_{i,i+2}$ for $i\in\ZZ$ and hence $c_{i,i-1}=-c_{i-1,i}=-d_{i-1}$. Then
\begin{enumerate}
    \item \label{eq:rowprop} 
    $(c_{i,j-1}, c_{i,j}) \emu(c_{j-1},d_j,d_{j-1}) = (c_{i,j}, c_{i,j+1})$ 
    \,for all $i\in \mathbb{Z}$ and $i\le j\le n+i+3$.
    \item $\emu(c_{i-1},d_i,d_{i-1})^T \begin{pmatrix} c_{i-1,k} \\ c_{i,k} 
    \end{pmatrix} = \begin{pmatrix} c_{i,k} \\ c_{i+1,k} \end{pmatrix}$
    \,for all $k\in \mathbb{Z}$ and $k-n-3\le i\le k$
\end{enumerate}

\end{Proposition}

\begin{proof}
We prove part (1), the proof of part (2) is similar. 

We first consider the cases $j=i$ and $j=i+1$ separately. 
For $j=i$ we have
$$(c_{i,i-1}, c_{i,i}) \emu(c_{i-1},d_i,d_{i-1}) = 
(-d_{i-1},0) \begin{pmatrix} 0 & -\frac{d_i}{d_{i-1}} \\ 1 & \frac{c_{i-1}}{d_{i-1}} \end{pmatrix}
= (0, d_i) = (c_{i,i}, c_{i,i+1})$$ 
as claimed. For $j=i+1$ we similarly get 
$$(c_{i,i}, c_{i,i+1}) \emu(c_{i},d_{i+1},d_{i}) = 
(0,d_i) \begin{pmatrix} 0 & -\frac{d_{i+1}}{d_{i}} \\ 1 & \frac{c_{i}}{d_{i}} \end{pmatrix}
= (d_i, c_i) = (c_{i,i+1}, c_{i,i+2}).$$ 
Now suppose $i+2\le j\le n+i+3$.
Then we consider the following complete adjacent $3\times 3$-submatrix of the extended frieze 
pattern
\[ M =
\begin{pmatrix}
c_{i,j-1} & c_{i,j} & c_{i,j+1} \\
c_{i+1,j-1} & c_{i+1,j} & c_{i+1,j+1} \\
c_{i+2,j-1} & c_{i+2,j} & c_{i+2,j+1}
\end{pmatrix}.
\]
The first two columns of $M$ cannot be linearly dependent because the upper left $2\times 2$-submatrix has determinant $d_i d_{j-1} \ne 0$, by Definition \ref{def:frieze} 
and Remark \ref{rem:frieze}.

But $\mathcal{C}$ is tame by assumption, hence the determinant of $M$ is zero, so
there are suitable numbers $s,t$ such that 
\[ M =
\begin{pmatrix}
c_{i,j-1} & c_{i,j} & sc_{i,j-1}+tc_{i,j} \\
c_{i+1,j-1} & c_{i+1,j} & sc_{i+1,j-1}+tc_{i+1,j} \\
c_{i+2,j-1} & c_{i+2,j} & sc_{i+2,j-1}+tc_{i+2,j}
\end{pmatrix}. 
\]
Now Equations ($E_{i,j}$) and ($E_{i,j-1}$) imply
\begin{eqnarray*}
d_i d_{j} &=& c_{i,i+1} c_{j,j+1} = c_{i+1,n+i+3} c_{j,j+1} \\ 
& = & c_{i,j}c_{i+1,j+1}-c_{i,j+1} c_{i+1,j} \\
& = & c_{i,j}(sc_{i+1,j-1}+tc_{i+1,j})-c_{i+1,j}(sc_{i,j-1}+tc_{i,j}) \\
&=& s(c_{i,j}c_{i+1,j-1} - c_{i+1,j}c_{i,j-1})  \\
& = & -s c_{i+1,n+i+3} c_{j-1,j} = -s d_i d_{j-1},
\end{eqnarray*}
and we conclude that $s=-\frac{d_{j}}{d_{j-1}}$.
Thus we see from the shape of the matrix $M$ 
that for fixed $j-1$, there is a $t_{j-1}=t$ such that for all $i$ we have
$$(c_{i,j-1}, c_{i,j}) \begin{pmatrix} 0 & -\frac{d_j}{d_{j-1}} \\
1 & t_{j-1} \end{pmatrix} = (c_{i,j}, c_{i,j+1}).
$$
In particular, this equation holds for $i=j-1$ and we get
\[ c_{j-1} = c_{j-1,j+1} = -\underbrace{c_{j-1,j-1}}_{=0} \frac{d_j}{d_{j-1}} + c_{j-1,j}t_{j-1} =
c_{j-1,j} t_{j-1} = d_{j-1} t_{j-1}, \]
hence $t_{j-1} = \frac{c_{j-1}}{d_{j-1}}$.
Altogether we obtain 
$$(c_{i,j-1}, c_{i,j}) \emu(c_{j-1},d_j,d_{j-1}) 
= (c_{i,j-1}, c_{i,j}) \begin{pmatrix} 0 & -\frac{d_j}{d_{j-1}} \\
1 & \frac{c_{j-1}}{d_{j-1}} \end{pmatrix} = (c_{i,j}, c_{i,j+1}),
$$
as claimed.
\end{proof}

As a consequence we can give a useful formula for determining the entries of
the frieze pattern with coefficients from the corresponding $\emu$-matrices.

\begin{Corollary} \label{cor:cij}
Let $R\subseteq \mathbb{C}$ be a subset.
Let $\mathcal{C}=(c_{i,j})$ be a tame frieze pattern with coefficients over $R$ of height
$n$, and as before set $d_i=c_{i,i+1}$, $c_i=c_{i,i+2}$ for $i\in\ZZ$.
Then we have
\begin{enumerate}
    \item
$\prod_{k=i}^{j} \emu(c_{k-1},d_k,d_{k-1})
= \frac{1}{d_{i-1}} \begin{pmatrix} -c_{i,j} & -c_{i,j+1} \\ c_{i-1,j} & c_{i-1,j+1} \end{pmatrix}$
for all $i\in \mathbb{Z}$ and $i-1\le j\le n+i+2$,
\item
$\prod_{k=1}^{n+3} \emu(c_{k-1},d_{k},d_{k-1}) = -\id$.
\end{enumerate}
\end{Corollary}

\begin{proof}
We consider the extended frieze pattern with coefficients as in Remark \ref{def:extfrieze}, 
in particular we set $c_{i,i-1}=-d_{i-1}$ for all $i\in \mathbb{Z}$.

$(1)$ From the first part of Proposition \ref{propagation_rev} we know that
\begin{eqnarray*}
\begin{pmatrix} -d_{i-1} & 0 \\ 0 & d_{i-1} \end{pmatrix}
\cdot \left( \prod_{k=i}^j \emu(c_{k-1},d_k,d_{k-1})\right)
& = & 
\begin{pmatrix} c_{i,i-1} & c_{i,i} \\ c_{i-1,i-1} & c_{i-1,i} \end{pmatrix}
\cdot \left( \prod_{k=i}^j \emu(c_{k-1},d_k,d_{k-1})\right) \\
& = & \begin{pmatrix} c_{i,j} & c_{i,j+1} \\ c_{i-1,j} & c_{i-1,j+1} \end{pmatrix}
\end{eqnarray*}
and the claim follows. 
\smallskip

$(2)$
By part $(1)$ for $i=1$ and $j=n+3$, we have
$$\prod_{k=1}^{n+3} \emu(c_{k-1},d_k,d_{k-1})
= \frac{1}{d_{0}} \begin{pmatrix} -c_{1,n+3} & -c_{1,n+4} \\ c_{0,n+3} & c_{0,n+4} \end{pmatrix}.$$
Now the claim follows since $c_{0,n+3}=0=c_{1,n+4}$ by Definition \ref{def:frieze}, $c_{0,n+4}=-c_{0,1}=-d_0$ by Remark \ref{def:extfrieze},
and $c_{1,n+3}=c_{0,1}=d_0$ by Equation (\ref{eq:glide}).
\end{proof}

As another main result of this section
we can now prove that the entries of a tame frieze pattern with coefficients 
are invariant under a glide reflection. This will become crucial in the next sections. 

\begin{Theorem} \label{thm:glide}
Let $R\subseteq \mathbb{C}$ be a subset.
Let $\mathcal{C}=(c_{i,j})$ be a tame frieze pattern with coefficients over $R$ of height
$n$. Then for all entries of $\mathcal{C}$ we have
$$c_{i,j} = c_{j,n+i+3}.
$$
\end{Theorem}

\begin{proof}
We consider the $i^{th}$ row and the $(n+i+3)^{rd}$ column in the extended
frieze pattern corresponding to $\mathcal{C}$:
$$\begin{array}{cccccccccc}
     &  &  &  & &   & & -d_{i-1}  & \\
    ~-d_{i-1}~ & ~0~ & ~d_i~ & ~\ldots~ & ~c_{i,j}~ & ~\ldots~ & ~c_{i,n+i+2}~ & ~0~ & ~-d_i~ &  \\
    &  &  &  & &   & & c_{i+1,n+i+3}   & 0 \\
    &  &  &  & &   & & c_{i+2,n+i+3}   & \\
    &  &  &  & &   & & \vdots  & \\
    &  &  &  & &   & & c_{j-1,n+i+3}  & \\
    &  &  &  & &   & & c_{j,n+i+3}  & \\
    &  &  &  & &   & & c_{j+1,n+i+3}  & \\
    &  &  &  & &   & & \vdots & \\
    &  &  &  & &   & & d_{n+i+2} & \\
    &  &  &  & &   & & 0 & \\
     &  &  &  & &   & & -d_{n+i+3} & \\
\end{array}
$$
We propagate along the $(n+i+3)^{rd}$ column using Proposition \ref{propagation_rev}\,(2) and get
\begin{eqnarray*}
\begin{pmatrix} c_{j,n+i+3} \\ c_{j+1,n+i+3} \end{pmatrix}
& = & \emu(c_{j-1},d_j,d_{j-1})^T\ldots \emu(c_{i},d_{i+1},d_{i})^T \emu(c_{i-1},d_{i},d_{i-1})^T
\begin{pmatrix} -d_{i-1} \\ 0 \end{pmatrix} \\
& = & \left( \prod_{k=i}^j \emu(c_{k-1},d_k,d_{k-1}) \right)^T 
\begin{pmatrix} -d_{i-1} \\ 0 \end{pmatrix}. 
\end{eqnarray*}
This implies that
$$
c_{j,n+i+3} =  -d_{i-1} \left( \left( \prod_{k=i}^j \emu(c_{k-1},d_k,d_{k-1})\right)^T \right)_{1,1} 
= -d_{i-1}\left( \prod_{k=i}^j \emu(c_{k-1},d_k,d_{k-1}) \right)_{1,1} 
= c_{i,j}
$$
where the last equality holds by Corollary \ref{cor:cij}.
\end{proof}

\section{Ptolemy relations}
\label{sec:ptolemy}

In the previous section we 
have seen that every tame frieze pattern with coefficients
over some subset $R\subseteq \mathbb{C}$ 
satisfies a glide symmetry, see Theorem \ref{thm:glide}. More precisely, Theorem \ref{thm:glide} states that
in Figure \ref{fig:glide}
the red area forms a fundamental domain for the action of the glide symmetry;
that is, the blue row is the same as the rightmost column of the red area and the green column is
the same as the top row of the red area. Note that the indices of the entries in the red fundamental area
are in bijection with the edges and diagonals of a regular $(n+3)$-gon (viewed as pairs of vertices), 
with vertices labelled 
$1,2,\ldots,n+3$. This means that we can view every tame 
frieze pattern with coefficients of height $n$ over $R$ as a map on the edges and
diagonals of a regular $(n+3)$-gon with values in $R$.
\begin{figure}
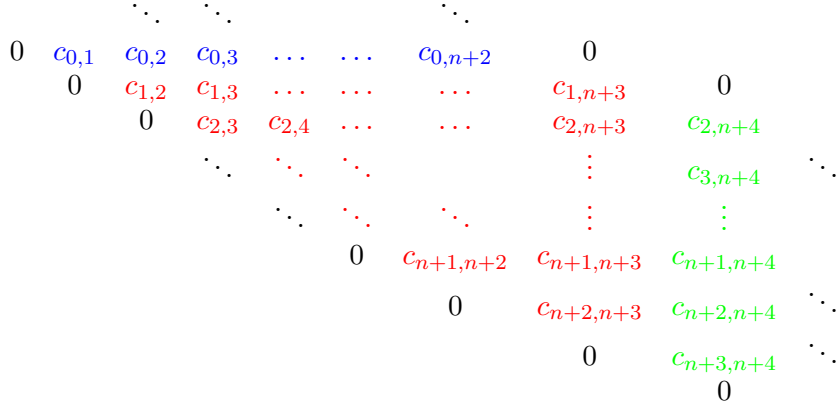

$$
\begin{array}{cccccccccc}
& ~~ & ~\ddots~ & ~\ddots~ & ~~ & ~~ & ~\ddots~ & ~~ & ~~ & \\ 
~0~& ~\blue{c_{0,1}}~ & ~\blue{c_{0,2}}~ & ~\blue{c_{0,3}}~ & ~\blue{\ldots}~ & ~\blue{\ldots}~ & ~\blue{c_{0,n+2}}~ & ~0~ & ~~ & \\ 
& ~0~ & ~\red{c_{1,2}}~ & ~\red{c_{1,3}}~ & ~\red{\ldots}~ & ~\red{\ldots}~ & ~\red{\ldots}~ & ~\red{c_{1,n+3}}~ & ~0~ & \\ 
& & ~0~ & ~\red{c_{2,3}}~ & ~\red{c_{2,4}}~ & ~\red{\ldots}~ & ~\red{\ldots}~ & ~\red{c_{2,n+3}}~ & ~\green{c_{2,n+4}}~ & \\
& &  & ~\ddots~ & ~\red{\ddots}~ & ~\red{\ddots}~ & ~~ & ~\red{\vdots}~ & ~\green{c_{3,n+4}}~ & ~\ddots~\\
& &  & ~~ & ~\ddots~ & ~\red{\ddots}~ & ~\red{\ddots}~ & ~\red{\vdots}~ & ~\green{\vdots}~ & ~~\\
& &  & ~~ & ~~ & ~0~ & ~\red{c_{n+1,n+2}}~ & ~\red{c_{n+1,n+3}}~ & ~\green{c_{n+1,n+4}}~ & ~~\\
& &  & ~~ & ~~ & ~~ & ~0~ & ~\red{c_{n+2,n+3}}~ & ~\green{c_{n+2,n+4}}~ & ~\ddots~\\
& &  & ~~ & ~~ & ~~ & ~~ & ~0~ & ~\green{c_{n+3,n+4}}~ & ~\ddots~\\
& &  & ~~ & ~~ & ~~ & ~~ & ~~ & ~0~ & ~~\\
\end{array}
$$
\caption{A fundamental domain for the glide symmetry of a frieze
with coefficients.\label{fig:glide}}
\end{figure}

\bigskip

\noindent
{\bf Convention:} 
We use the notion (tame) {\em frieze pattern with coefficients} for an infinite array 
as in Definition \ref{def:frieze} and the notion (tame) {\em frieze with coefficients} 
for a corresponding map from edges and diagonals of a regular polygon. 

\bigskip

The entries in a frieze (pattern) with coefficients are tightly connected by many remarkable
equations, in addition to the local conditions given in equations 
(\ref{eq:local}) in Definition \ref{def:frieze}. 

\begin{Definition} \label{def:ptolemy}
Let $\mathcal{C}=(c_{i,j})$ be a tame frieze with coefficients on a regular $m$-gon. 
For any indices $1\le i\le j\le k\le \ell\le m$
the corresponding Ptolemy relation is the equation
\begin{equation*}\tag{$E_{i,j,k,\ell}$}\label{eq:ptolemy}
c_{i,k} c_{j,\ell} = c_{i,\ell} c_{j,k} + c_{i,j} c_{k,\ell}
\end{equation*}
\end{Definition}

\begin{Remark} \label{rem:ptolemy}
\begin{enumerate}
    \item By Definition \ref{def:frieze}, the entries $c_{i,i}$ are zero for
    all $i\in \mathbb{Z}$. This implies that the equations (\ref{eq:ptolemy})
    always hold if there are equalities among the numbers $i,j,k,\ell$. 
    \item If $i< j< k< \ell$ then the Ptolemy relation (\ref{eq:ptolemy})
    can be visualized as follows 
    \begin{center}
  \begin{tikzpicture}[auto]
    \node[name=s, draw, shape=regular polygon, regular polygon sides=500, minimum size=4cm] {};
    \draw[thick] (s.corner 60) to (s.corner 180);
    \draw[thick] (s.corner 180) to (s.corner 300);
    \draw[thick] (s.corner 300) to (s.corner 400);
    \draw[thick] (s.corner 400) to (s.corner 60);
    \draw[thick] (s.corner 60) to (s.corner 300);
    \draw[thick] (s.corner 180) to (s.corner 400);
    
    \draw[shift=(s.corner 60)]  node[above]  {{\small $i$}};
    \draw[shift=(s.corner 180)]  node[left]  {{\small $j$}};
    \draw[shift=(s.corner 300)]  node[below]  {{\small $k$}};
    \draw[shift=(s.corner 400)]  node[right]  {{\small $\ell$}};
   \end{tikzpicture}
\end{center}
The Ptolemy relation asks that in the quadrilateral,
the product of the labels 
on the diagonals equals the sum of the products of labels of opposite
sides; cf.\ Ptolemy's theorem from Elementary Euclidean Geometry.
\item The local conditions for a frieze pattern from Definition \ref{def:frieze}
are a special case of Ptolemy relations. Namely, the local condition 
(\ref{eq:local}) is equal to the Ptolemy relation
$(E_{i,i+1,j,j+1})$ (use that $c_{i,i+1}=c_{i+1,n+i+3}$ by Remark \ref{rem:frieze}).
\end{enumerate}
\end{Remark}

\begin{Theorem} \label{thm:ptolemy}
Every tame frieze with coefficients
over some subset $R\subseteq \mathbb{C}$
satisfies all Ptolemy relations.
\end{Theorem}

\begin{proof}
Let $\mathcal{C}$ be a tame frieze on a regular $m$-gon, where $m\ge 3$.
Take any four vertices $1\le i\le j\le k\le \ell\le m$ of the regular $m$-gon.
By Corollary \ref{cor:cij} we have that
\[ M_{i+1,j} := \prod_{k=i+1}^{j} \emu(c_{k-1},d_{k},d_{k-1})
= \frac{1}{d_i} \begin{pmatrix} -c_{i+1,j} & -c_{i+1,j+1} \\ c_{i,j} & c_{i,j+1} \end{pmatrix}. \]
Using that $M_{i+1,k} = M_{i+1,j}M_{j+1,k}$, $M_{j+1,\ell} = M_{j+1,k}M_{k+1,\ell}$, 
and $M_{i+1,\ell} = M_{i+1,j} M_{j+1,k}M_{k+1,\ell}$ we get from the $(2,1)$-entries of the 
matrices that
\begin{eqnarray}
\label{me1} \frac{1}{d_i} c_{i,k} &=& \frac{1}{d_id_j}(c_{i,j+1} c_{j,k} - c_{j+1,k} c_{i,j}), \\
\label{me2} \frac{1}{d_j} c_{j,\ell} &=& \frac{1}{d_j d_k} ( c_{j,k+1} c_{k,l} - c_{k+1,\ell} c_{j,k}), \\
\label{me3} \frac{1}{d_i} c_{i,\ell} &=& \frac{1}{d_i d_j d_k} (c_{i,j+1} c_{j,k+1} c_{k,\ell} - c_{i,j+1} c_{k+1,\ell} c_{j,k} + c_{j+1,k} c_{k+1,\ell} c_{i,j} - c_{j+1,k+1} c_{i,j} c_{k,\ell}).
\end{eqnarray}
With these three equations and $(E_{j,k})$ we conclude:
\begin{eqnarray*}
&& \frac{1}{d_i d_j} (c_{i,j} c_{k,\ell} + c_{j,k} c_{i,\ell} - c_{i,k} c_{j,\ell}) \\
&\stackrel{(\ref{me1}), (\ref{me2})}{=}& \frac{1}{d_i d_j} (c_{i,j} c_{k,\ell} + c_{j,k} c_{i,\ell}) - \frac{1}{d_i d_j^2 d_k}(c_{i,j+1} c_{j,k} - c_{j+1,k} c_{i,j})( c_{j,k+1} c_{k,l} - c_{k+1,\ell} c_{j,k})\\
&\stackrel{(\ref{me3})}{=}& \frac{1}{d_i d_j} (c_{i,j} c_{k,\ell} + c_{j,k} c_{i,\ell}) - \frac{1}{d_i d_j^2 d_k}(c_{i,j+1} c_{j,k} - c_{j+1,k} c_{i,j})( c_{j,k+1} c_{k,l} - c_{k+1,\ell} c_{j,k}) + \\
&& \frac{1}{d_j} c_{j,k} (\frac{1}{d_i d_j d_k} \left(c_{i,j+1} c_{j,k+1} c_{k,\ell} - c_{i,j+1} c_{k+1,\ell} c_{j,k} + c_{j+1,k} c_{k+1,\ell} c_{i,j} - c_{j+1,k+1} c_{i,j} c_{k,\ell}\right)\\ && - \frac{1}{d_i} c_{i,\ell}) \\
&=& \frac{1}{d_i d_j} c_{i,j} c_{k,\ell} \left(1-\frac{1}{d_j d_k} \left(c_{j,k}c_{j+1,k+1} - c_{j,k+1} c_{j+1,k}\right)\right) \stackrel{(E_{j,k})}{=} 0.
\end{eqnarray*}
Hence we obtain $c_{i,j} c_{k,\ell} + c_{j,k} c_{i,\ell} = c_{i,k} c_{j,\ell}$,
that is, the Ptolemy relation $(E_{i,j,k,\ell})$ holds. 
\end{proof}

\begin{Remark}
We have now seen that a tame frieze pattern with coefficients as in Definition 
\ref{def:frieze} is basically the same as a map on edges and diagonals of a polygon
satisfying all Ptolemy relations. This means that indeed 
the local condition of a frieze pattern with
coefficients produces the cluster variables of a Ptolemy cluster algebra with coefficients; 
see for instance \cite[Section 1]{Schiffler} for a description of the Ptolemy cluster
algebra (aka cluster algebra of Dynkin type $A$ with non-trivial coefficients).
\end{Remark}


\section{Frieze patterns from subpolygons}
\label{sec:subpolygons}

In the previous section we have seen that a tame frieze pattern with coefficients
of height $n$ can be seen as a map from edges and diagonals of an $(n+3)$-gon
such that all Ptolemy relations are satisfied.

From now on we consider friezes with coefficients over positive integers. 
For classic friezes over $\mathbb{N}$ there is a beautiful combinatorial 
description via triangulations of polygons \cite{CC73}.
\smallskip

\noindent{\bf Convention:} From now on we use the notion {\em classic 
Conway-Coxeter frieze (pattern)} for a classic frieze (pattern) over the positive
integers $\mathbb{N}$.
\smallskip

Our aim is to connect the more general theory of friezes with coefficients
over $\mathbb{N}$ 
with the theory of classic Conway-Coxeter friezes. 

From the viewpoint of polygons there is an obvious way to obtain friezes
with coefficients from classic Conway-Coxeter friezes. Namely, take any classic Conway-Coxeter 
frieze $\mathcal{C}$; this is given as a map $f_{\mathcal{C}}$ from edges and 
diagonals on a polygon to $\mathbb{N}$. Now cut out any subpolygon
and restrict the map $f_{\mathcal{C}}$ to this subpolygon. Clearly, the restricted
map still satisfies all Ptolemy relations of the subpolygon, that is, 
the restriction yields a frieze with coefficients.

\begin{Example} \label{ex:1122_part1}
We consider the following triangulation of a hexagon:

\begin{center}
  \begin{tikzpicture}[auto]
    \node[name=s, draw, shape=regular polygon, regular polygon sides=6, minimum size=3cm] {};
    \draw[thick] (s.corner 2) to (s.corner 4);
    \draw[thick] (s.corner 2) to (s.corner 5);
    \draw[thick] (s.corner 2) to (s.corner 6);
   \end{tikzpicture}
\end{center}
The corresponding Conway-Coxeter frieze 
has for instance the following values on the subpolygon highlighted 
in red:

\begin{center}
  \begin{tikzpicture}[auto]
    \node[name=s, draw, shape=regular polygon, regular polygon sides=6, minimum size=3cm] {};
    \draw[dash dot,red] (s.corner 1) to (s.corner 3);
    \draw[red] (s.corner 3) to (s.corner 5);
    \draw[dash dot,red] (s.corner 2) to (s.corner 5);
    \draw[red] (s.corner 1) to (s.corner 5);
    \draw[red] (s.corner 1) to (s.corner 2);
    \draw[red] (s.corner 2) to (s.corner 3);
    
    \draw[shift=(s.side 1)]  node[above]  {{\small {\red 1}}};    
    \draw[shift=(s.corner 2)]  node[below left=12pt]  {{\small {\red 1}}};
    \draw[shift=(s.corner 4)]  node[above right=5pt]  {{\small {\red 2}}};
    \draw[shift=(s.corner 6)]  node[left=10pt]  {{\small {\red 2}}};
    \draw[shift=(s.corner 6)]  node[left=42pt]  {{\small {\red 1}}};
    \draw[shift=(s.corner 1)]  node[below left=15pt]  {{\small {\red 4}}};
   \end{tikzpicture}
\end{center}
This yields a frieze with coefficients with boundary sequence $(1,1,2,2)$,
and diagonal values $1,4$ as indicated in the figure. 
\end{Example}

\medskip
The fundamental question then is: which friezes with coefficient over $\mathbb{N}$ can be 
obtained from classic Conway-Coxeter friezes by cutting out subpolygons?
\medskip

The following result gives some restrictions for the smallest case
of triangles.

\begin{Lemma}\label{lem:gcd}
Let $\mathcal{C}=(c_{i,j})$ be a classic Conway-Coxeter frieze and $i \le j \le k$.
Then the greatest common divisor of any two of the numbers
$c_{i,j},c_{j,k},c_{i,k}$ divides the third number.
In particular, $\gcd(c_{i,j}, c_{j,k})=\gcd(c_{j,k}, c_{k,i})=\gcd(c_{i,j}, c_{k,i})$.
\end{Lemma}

\begin{proof}
By symmetry and relabelling it suffices to show 
that $\gcd(c_{j,k},c_{i,k})$ divides $c_{i,j}$.

If $i=j$ then the assertion holds since $c_{i,i}=0$. If $j=k$ then 
$\gcd(c_{j,k},c_{i,k})=c_{i,k}=c_{i,j}$ and the claim follows. If $j=k-1$ then
$\gcd(c_{j,k},c_{i,k})=\gcd(1,c_{i,k})=1$ which clearly divides $c_{i,j}$. 

So from now on we can assume that $i<j<k-1$. Then the Ptolemy relation
for the crossing diagonals $(i,k-1)$ and $(j,k)$ implies that
$$c_{i,k-1}c_{j,k} = c_{j,k-1}c_{i,k} + c_{i,j}c_{k-1,k}
= c_{j,k-1}c_{i,k} + c_{i,j} .
$$
Therefore, any common divisor of $c_{j,k}$ and $c_{i,k}$
divides $c_{i,j}$ and the claim follows.
\end{proof}

\begin{Example} \label{ex:1122_part2}
\begin{enumerate}
    \item We have seen in Example \ref{ex:frieze} that any frieze with 
coefficients of height 0 over $\mathbb{N}$ is given by three 
numbers $a,b,c$; these are the values attached to the edges of 
the corresponding triangle. Lemma \ref{lem:gcd} implies that
such a triangle can only be cut out of a classic Conway-Coxeter frieze 
if the greatest common divisor of two of $a,b,c$ divides the third. 
For instance, a triangle with values $1,2,2$ can not come from
a classic Conway-Coxeter frieze. 

In the next section we will consider triangles in more detail and 
will obtain a complete characterisation for which triangles with
triples $a,b,c$ can
be cut out of a classic Conway-Coxeter frieze.
\item Consider the following frieze with coefficients on a square:

\begin{center}
  \begin{tikzpicture}[auto]
    \node[name=s, draw, shape=regular polygon, regular polygon sides=4, minimum size=3cm] {};
    \draw[thick] (s.corner 1) to (s.corner 3);
    \draw[thick] (s.corner 2) to (s.corner 4);
    
    \draw[shift=(s.side 1)]  node[above]  {{\small {\blue 1}}};    
    \draw[shift=(s.side 2)]  node[left]  {{\small {\blue 1}}};
    \draw[shift=(s.side 3)]  node[below]  {{\small {\blue 2}}};    
    \draw[shift=(s.side 4)]  node[right]  {{\small {\blue 2}}};
    \draw[shift=(s.side 3)]  node[above left=7pt]  {{\small {\blue 1}}};    
    \draw[shift=(s.side 1)]  node[below left=7pt]  {{\small {\blue 4}}};
   \end{tikzpicture}
\end{center}

This square can not be cut out of a classic Conway-Coxeter frieze because it
contains a triangle with values $1,2,2$, which can not come from a classic 
Conway-Coxeter frieze by part (1).
\end{enumerate}
\end{Example}

We have seen in Examples \ref{ex:1122_part1} and \ref{ex:1122_part2}
that whether a frieze with coefficients can be cut out of a classic Conway-Coxeter
frieze does not only depend on the boundary sequence. 
\medskip

The next result shows that triangles can be cut out of a classic Conway-Coxeter
frieze if one of the values on the edges is 1 and the other two values are coprime. 
Note that the coprimeness condition has to be satisfied in this case by Lemma \ref{lem:gcd}. The proof
of the lemma exhibits a close connection to the Euclidean Algorithm.

\begin{Lemma}\label{accordion}
Let $a,b\in \mathbb{Z}_{\ge 0}$ with $\gcd(a,b)=1$. Then there exists a classic Conway-Coxeter frieze $\mathcal{C}=(c_{i,j})$ and a $k\ge 1$
such that $c_{1,k}=a$, $c_{k,k+1}=1$, and $c_{1,k+1}=b$.
\end{Lemma}

\begin{proof}
Note that by assumption not both of $a,b$ can be zero, since $\gcd(0,0)=0$. 
If one of them is zero, say $a=0$, then $b=\gcd(0,b)=\gcd(a,b)=1$ and the assertion holds with $k=1$; in fact, use the unique classic Conway-Coxeter frieze
on a triangle.

So from now on we can assume that $a,b\in \mathbb{N}$. We show how the 
Euclidean Algorithm leads to a triangulation of a polygon, and hence
to a classic Conway-Coxeter frieze, with the required values.

First we perform the Euclidean Algorithm, where we can assume that $a\ge b$.
Set $r_0=a$ and $r_1=b$, then
\begin{eqnarray*}
a & = & q_1 b + r_2\\
b & = & q_2r_2 + r_3\\
r_2 & = & q_3r_3+r_4\\
 & \vdots & \\
r_{\ell-2} & = & q_{\ell-1} r_{\ell-1} + r_{\ell} \\
r_{\ell-1} & = & q_{\ell} r_{\ell} + r_{\ell+1} \\
r_{\ell} & = & q_{\ell+1}r_{\ell+1} + 0
\end{eqnarray*}
Note that by assumption we have $r_{\ell+1} = \gcd(a,b)=1$ and
$q_{\ell+1}=r_{\ell}$.

We now describe how to obtain from this data a suitable triangulation of 
a polygon. 

Consider the usual number line, and the integral points as vertices. 
Start with the edge connecting $1$ and $2$. Then draw $q_{\ell+1}$ 
arcs from 2 to the left to the next available vertices, i.e. to the vertices $0,-1,\ldots, -q_{\ell+1}+1$. Then from vertex $-q_{\ell+1}+1$ draw $q_{\ell}$
arcs to the right to the next available vertices. Continue 
in this alternating fashion by drawing $q_{\ell-1},q_{\ell-2},\ldots,
q_3,q_2,q_1$ arcs to the left and right, respectively. 

This gives a triangulation of a polygon whose vertices are the integers
which are attached to one of the arcs described above. (The rest of the
number line is now disregarded.) Note that this polygon is an
$m:=(2+\sum_{j=1}^{\ell+1} q_j)$-gon. We keep the label for vertex 1, but 
then number the vertices of this polygon as usual consecutively
by $1,2,\ldots,m$.

We claim that the classic Conway-Coxeter frieze to this triangulation
of the $m$-gon has the desired properties.
For this, we use a well-known combinatorial algorithm for computing
arbitrary frieze entries from the triangulation, see \cite{BCI}, starting
from the vertex $1$. The values attached to the vertices by this
algorithm are then the
frieze entries $c_{1,j}$. 

The endpoint of the last of $q_{\ell+1}$ arcs to the left gets the label
$q_{\ell+1}=r_{\ell}$. Then the endpoints of the $q_{\ell}$
arcs to the right get the labels $r_{\ell}+1, 2r_{\ell}+1,\ldots,
q_{\ell} r_{\ell}+1= r_{\ell-1}$. Then the last of the endpoints of the 
$q_{\ell-1}$ arcs to the left gets the label 
$q_{\ell-1}r_{\ell-1}+r_{\ell} = r_{\ell-2}$. Eventually, we get two
consecutive vertices $k+1$ and $k$ in the polygon which get assigned the
values $r_1=b$ and $r_0=a$. 

This means that we have constructed a classic Conway-Coxeter frieze 
$\mathcal{C} = (c_{i,j})$ with $c_{k,k+1}=1$, $c_{1,k}=a$ and $c_{1,k+1}=b$. 
\end{proof}

\section{Triangles in Conway-Coxeter friezes}
\label{sec:triangles}

We have seen in Section \ref{sec:subpolygons} that cutting out any subpolygon 
of a classic Conway-Coxeter frieze yields a frieze with coefficients over 
$\mathbb{N}$. 
In this section, we give an explicit description of all possible values on 
the edges of triangles cut out of classic Conway-Coxeter friezes.
Recall that we have already obtained some results on triples 
$(a,b,c)$ appearing as labels of a triangle cut out of a classic Conway-Coxeter frieze.
Namely, the greatest common divisor of any two of the numbers divides the
third (see Lemma \ref{lem:gcd}), and for any pair $a,b$ of coprime 
natural numbers, the triple $(a,1,b)$ appears as a triangle in some classic 
Conway-Coxeter frieze (see Lemma \ref{accordion}). 

However, the complete classification 
of possible triples is rather subtle and will be the main topic of this
section; see Theorem \ref{mres} below for a precise statement of our 
main result.

The following auxiliary result shows that given three vertices of an 
$n$-gon, any triangulation of the $n$-gon
contains a triangle which separates the given points. 

\begin{Lemma}\label{inscribed_triangle}
Let $\mathcal{C}=(c_{i,j})$ be a classic Conway-Coxeter frieze on an $n$-gon. 
If $1\le i < j < k \le n$, then there exist $i',j',k'$ with $i\le i' \le j \le j' \le k \le k'$ or $k'\le i\le i' \le j \le j' \le k$ such that
\[ c_{i',j'} = c_{j',k'} = c_{k',i'} = 1. \]
\end{Lemma}
\begin{proof}
Let $\mathcal{T}$ be the triangulation of an $n$-gon corresponding
to the classic Conway-Coxeter frieze $\mathcal{C}$. 
Any diagonal of $\mathcal{T}$
divides the $n$-gon into two subpolygons. 

If each diagonal of $\mathcal{T}$ has the property that 
the three vertices $i,j,k$ are contained in one of the two
subpolygons then the triangle given by $i,j,k$ is a triangle 
of the triangulation $\mathcal{T}$ and we are done
by choosing $i=i'$, $j=j'$ and $k=k'$. 

So we can choose a diagonal $(u,v)$ of $\mathcal{T}$ which separates
the three vertices. By symmetry and possibly relabelling the vertices
we can assume that we have the situation as given in the following figure:

\begin{center}
  \begin{tikzpicture}[auto]
    \node[name=s, draw, shape=regular polygon, regular polygon sides=500, minimum size=3cm] {};
    \draw[thick] (s.corner 40) to (s.corner 300);
    
    \draw[shift=(s.corner 40)]  node[above]  {{\small $u$}};
    \draw[shift=(s.corner 300)]  node[below]  {{\small $v$}};
    \draw[shift=(s.corner 90)]  node[left]  {{\small $i$}};
    \draw[shift=(s.corner 240)]  node[below]  {{\small $j$}};
    \draw[shift=(s.corner 400)]  node[right]  {{\small $k$}};
   \end{tikzpicture}
\end{center}
Moreover, we choose the diagonal $(u,v)\in \mathcal{T}$
such that $|i-u|$ and $|v-j|$
are minimal. That is, of all the diagonals of $\mathcal{T}$
separating the three vertices in this way 
we choose the one closest to the 
diagonal $(i,j)$. Note that the latter need not be in 
$\mathcal{T}$, but the cases $i=u$ and $j=v$ are allowed. 

If $j\neq i+1$ then the diagonal $(u,v)$ of $\mathcal{T}$ is
part of a triangle in the left subpolygon. By minimality of
$|i-u|$ and $|j-v|$ the third endpoint $w$ of this triangle 
must satisfy $i+1\le w \le j-1$. Now set $i'=w$, $j'=v$,
$k'=u$, and we are done. 

This leaves us with the case that $j=i+1$. But then the edge
$(i,i+1)$ is part of a triangle of $\mathcal{T}$ and it is 
not hard to check that the assertion of the lemma holds in this 
case. 
\end{proof}

We now come to a crucial reformulation of the problem of describing
the triangles which appear as subpolygons in classic Conway-Coxeter friezes. 
Namely, we are going to show that such triangles correspond 
(but not bijectively!)\ to certain tuples of coprime pairs of numbers.

To understand the following explanations, 
Figure \ref{fig_triangle} might be helpful.
In particular, the Ptolemy relations required in the proof of Proposition \ref{prop:Delta1} are easy to read from the 
picture.

\begin{Definition}  \label{def:S}
Let
\[ S:=\{ (a_1,a_2,b_1,b_2,c_1,c_2)\in \ZZ^6 \mid \gcd(a_1,a_2)=\gcd(b_1,b_2)=\gcd(c_1,c_2)=1 \}\]
and $S_{\ge 0}:=S\cap \ZZ_{\ge 0}^6$.
Moreover, let $\Delta : S \rightarrow \ZZ^3$ be the map
\[ \Delta((a_1,a_2,b_1,b_2,c_1,c_2)):=(b_1 c_1+b_1 c_2+b_2 c_2, a_1 c_1+a_2 c_1+a_2 c_2, a_1 b_1+a_1 b_2+a_2 b_2).\]
\end{Definition}

\begin{figure}
\begingroup%
  \makeatletter%
  \providecommand\color[2][]{%
    \errmessage{(Inkscape) Color is used for the text in Inkscape, but the package 'color.sty' is not loaded}%
    \renewcommand\color[2][]{}%
  }%
  \providecommand\transparent[1]{%
    \errmessage{(Inkscape) Transparency is used (non-zero) for the text in Inkscape, but the package 'transparent.sty' is not loaded}%
    \renewcommand\transparent[1]{}%
  }%
  \providecommand\rotatebox[2]{#2}%
  \newcommand*\fsize{\dimexpr\f@size pt\relax}%
  \newcommand*\lineheight[1]{\fontsize{\fsize}{#1\fsize}\selectfont}%
  \ifx\svgwidth\undefined%
    \setlength{\unitlength}{330.16299717bp}%
    \ifx\svgscale\undefined%
      \relax%
    \else%
      \setlength{\unitlength}{\unitlength * \real{\svgscale}}%
    \fi%
  \else%
    \setlength{\unitlength}{\svgwidth}%
  \fi%
  \global\let\svgwidth\undefined%
  \global\let\svgscale\undefined%
  \makeatother%
  \begin{picture}(1,1.07961399)%
    \lineheight{1}%
    \setlength\tabcolsep{0pt}%
    \put(0,0){\includegraphics[width=\unitlength,page=1]{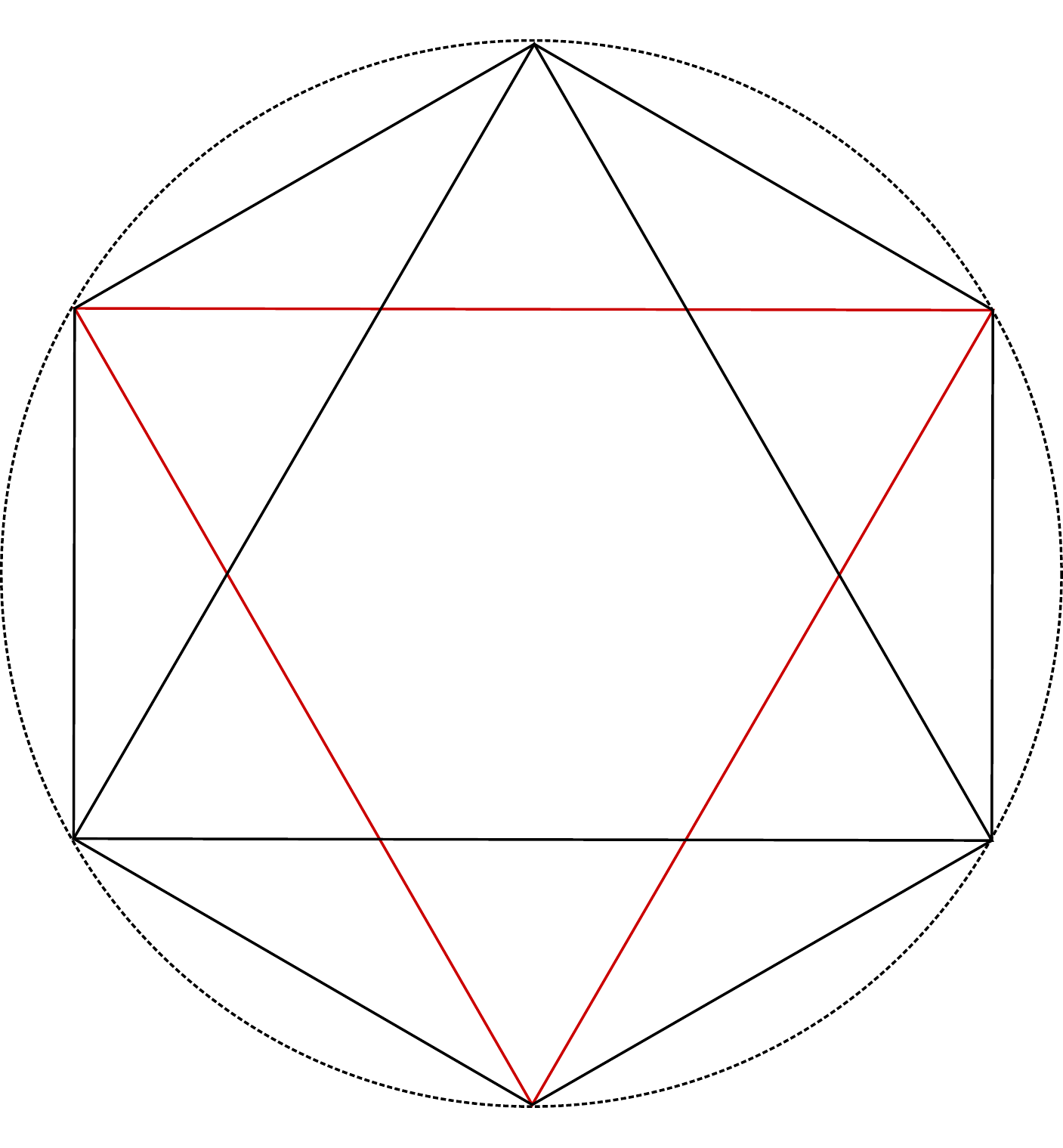}}%
    \put(0.0357507,0.26840538){\color[rgb]{0,0,0}\makebox(0,0)[lt]{\lineheight{1.25}\smash{\begin{tabular}[t]{l}$i$\end{tabular}}}}%
    \put(0.49137077,1.05711148){\color[rgb]{0,0,0}\makebox(0,0)[lt]{\lineheight{1.25}\smash{\begin{tabular}[t]{l}$j$\end{tabular}}}}%
    \put(0.94353929,0.26667969){\color[rgb]{0,0,0}\makebox(0,0)[lt]{\lineheight{1.25}\smash{\begin{tabular}[t]{l}$k$\end{tabular}}}}%
    \put(0.16173652,0.42200456){\color[rgb]{0,0,0}\makebox(0,0)[lt]{\lineheight{1.25}\smash{\begin{tabular}[t]{l}$a$\end{tabular}}}}%
    \put(0.56730746,0.8655439){\color[rgb]{0,0,0}\makebox(0,0)[lt]{\lineheight{1.25}\smash{\begin{tabular}[t]{l}$b$\end{tabular}}}}%
    \put(0.7623267,0.30119621){\color[rgb]{0,0,0}\makebox(0,0)[lt]{\lineheight{1.25}\smash{\begin{tabular}[t]{l}$c$\end{tabular}}}}%
    \put(0.6812126,0.17693635){\color[rgb]{0,0,0}\makebox(0,0)[lt]{\lineheight{1.25}\smash{\begin{tabular}[t]{l}$a_1$\end{tabular}}}}%
    \put(0.89003838,0.52382886){\color[rgb]{0,0,0}\makebox(0,0)[lt]{\lineheight{1.25}\smash{\begin{tabular}[t]{l}$a_2$\end{tabular}}}}%
    \put(0.7071001,0.88107646){\color[rgb]{0,0,0}\makebox(0,0)[lt]{\lineheight{1.25}\smash{\begin{tabular}[t]{l}$c_1$\end{tabular}}}}%
    \put(0.28427079,0.88452812){\color[rgb]{0,0,0}\makebox(0,0)[lt]{\lineheight{1.25}\smash{\begin{tabular}[t]{l}$c_2$\end{tabular}}}}%
    \put(0.0754449,0.53763551){\color[rgb]{0,0,0}\makebox(0,0)[lt]{\lineheight{1.25}\smash{\begin{tabular}[t]{l}$b_1$\end{tabular}}}}%
    \put(0.27046414,0.18038796){\color[rgb]{0,0,0}\makebox(0,0)[lt]{\lineheight{1.25}\smash{\begin{tabular}[t]{l}$b_2$\end{tabular}}}}%
    \put(0,0){\includegraphics[width=\unitlength,page=2]{triangle_generic.pdf}}%
    \put(0.3274166,0.6291047){\color[rgb]{0,0,0}\makebox(0,0)[lt]{\lineheight{1.25}\smash{\begin{tabular}[t]{l}$a_1+a_2$\end{tabular}}}}%
    \put(0.60700166,0.62047552){\color[rgb]{0,0,0}\makebox(0,0)[lt]{\lineheight{1.25}\smash{\begin{tabular}[t]{l}$b_1+b_2$\end{tabular}}}}%
    \put(0.42406327,0.38058461){\color[rgb]{0,0,0}\makebox(0,0)[lt]{\lineheight{1.25}\smash{\begin{tabular}[t]{l}$c_1+c_2$\end{tabular}}}}%
    \put(0.42578913,0.75854225){\color[rgb]{0,0,0}\makebox(0,0)[lt]{\lineheight{1.25}\smash{\begin{tabular}[t]{l}$1$\end{tabular}}}}%
    \put(0.71745492,0.46687634){\color[rgb]{0,0,0}\makebox(0,0)[lt]{\lineheight{1.25}\smash{\begin{tabular}[t]{l}$1$\end{tabular}}}}%
    \put(0.30670657,0.40129469){\color[rgb]{0,0,0}\makebox(0,0)[lt]{\lineheight{1.25}\smash{\begin{tabular}[t]{l}$1$\end{tabular}}}}%
    \put(0.02712151,0.7878813){\color[rgb]{0,0,0}\makebox(0,0)[lt]{\lineheight{1.25}\smash{\begin{tabular}[t]{l}$i'$\end{tabular}}}}%
    \put(0.93836177,0.79305893){\color[rgb]{0,0,0}\makebox(0,0)[lt]{\lineheight{1.25}\smash{\begin{tabular}[t]{l}$j'$\end{tabular}}}}%
    \put(0.48791916,0.00435299){\color[rgb]{0,0,0}\makebox(0,0)[lt]{\lineheight{1.25}\smash{\begin{tabular}[t]{l}$k'$\end{tabular}}}}%
  \end{picture}%
\endgroup%
\caption{A triangle $(a,b,c)$ and its surroundings.\label{fig_triangle}}
\end{figure}

The following result explains the relevance of the sets and maps in Definition 
\ref{def:S} for our purposes, namely every triple of labels of a
triangle cut out of a classic Conway-Coxeter
frieze is in the image of the map $\Delta$.

\begin{Proposition} \label{prop:Delta1}
Let $\mathcal{C}=(c_{i,j})$ be a classic Conway-Coxeter frieze. 
If $i < j < k$,
then there exists a tuple $(a_1,a_2,b_1,b_2,c_1,c_2) \in S_{\ge 0}$ such that
$(c_{i,j}, c_{j,k}, c_{k,i}) = \Delta((a_1,a_2,b_1,b_2,c_1,c_2))$.
\end{Proposition}

\begin{proof}
Choose $i',j',k'$ as in Lemma \ref{inscribed_triangle}, i.e.
\begin{equation}\label{c_ones}
c_{i',j'} = c_{j',k'} = c_{k',i'} = 1.
\end{equation}
By Lemma \ref{lem:gcd},
\[ \gcd(c_{k,k'},c_{j',k})=\gcd(c_{i,i'},c_{k',i})=\gcd(c_{j,j'},c_{i',j}) = 1. \]
Using the Ptolemy relations (Theorem \ref{thm:ptolemy} and also Remark 
\ref{rem:ptolemy}\,(1))) and Equation (\ref{c_ones}), we obtain
\begin{eqnarray*}
c_{i',k} &=& c_{k,k'} + c_{j',k}, \\
c_{j',i} &=& c_{i,i'} + c_{k',i}, \\
c_{k',j} &=& c_{j,j'} + c_{i',j}.
\end{eqnarray*}
Again by the Ptolemy relations we now obtain
\begin{eqnarray*}
c_{i,j} &=& c_{i,i'} c_{k',j} + c_{k',i} c_{i',j} = c_{i,i'} c_{j,j'} + 
c_{i,i'} c_{i',j} + c_{k',i} c_{i',j}, \\
c_{j,k} &=& c_{i',k} c_{j,j'} + c_{j',k} c_{i',j} = 
c_{k,k'} c_{j,j'} + c_{j',k} c_{j,j'} + c_{j',k} c_{i',j} , \\
c_{k,i} &=& c_{k,k'} c_{j',i} + c_{j',k} c_{k',i} = 
c_{k,k'} c_{i,i'} + c_{k,k'} c_{k',i} +  c_{j',k} c_{k',i},
\end{eqnarray*}
and hence
$\Delta((c_{k,k'},c_{j',k},c_{i,i'},c_{k',i},c_{j,j'},c_{i',j})) = (c_{i,j}, c_{j,k}, c_{k,i})$,
as desired.
\end{proof}

We now show a converse to Proposition \ref{prop:Delta1}, namely that
every triple in the image of $S_{\ge 0}$ under the map $\Delta$ actually appears as labels of
a triangle cut out of some classic Conway-Coxeter frieze.

\begin{Proposition} \label{prop:Delta2}
Let $(a_1,a_2,b_1,b_2,c_1,c_2)\in S_{\ge 0}$.
Then there exists a classic Conway-Coxeter frieze $\mathcal{C}=(c_{i,j})$ and $i \le j \le k$ such that
$(c_{i,j}, c_{j,k}, c_{k,i}) = \Delta((a_1,a_2,b_1,b_2,c_1,c_2))$.
\end{Proposition}

\begin{proof}
By Lemma \ref{accordion}, there are three classic Conway-Coxeter friezes $\tilde{\mathcal{C}}=(\tilde c_{i,j})$, $\mathcal{C'}=(c'_{i,j})$, $\mathcal{C''}=(c''_{i,j})$
such that
\begin{eqnarray*}
&& \tilde c_{1,2}=1, \quad \tilde c_{2,\ell}=a_1, \quad \tilde c_{\ell,1}=a_2, \\
&& c'_{1,2}=1, \quad  c'_{2,\ell'}=b_1, \quad  c'_{\ell',1}=b_2, \\
&& c''_{1,2}=1, \quad  c''_{2,\ell''}=c_1, \quad  c''_{\ell'',1}=c_2
\end{eqnarray*}
for some $\ell,\ell',\ell''$.
Let $\tilde{\mathcal{T}},\mathcal{T'},\mathcal{T''}$ be the corresponding triangulations 
of some regular polygons.
We glue these three triangulations together in such a way that the edges with
labels $\tilde c_{1,2}$, $c'_{1,2}$, $c''_{1,2}$ become edges of an inner triangle and 
obtain a new triangulation $\mathcal{T}$. In the classic Conway-Coxeter frieze $\mathcal{C}$
corresponding to $\mathcal{T}$, this inner triangle carries labels
$c_{i',j'}=c_{j',k'}=c_{k',i'}=1$ for some $i',j',k'$. Moreover, there are $i,j,k$ such that
\[ c_{i,i'} = b_1, \quad c_{k',i} = b_2, \quad c_{j,j'} = c_1, \quad c_{i',j} = c_2, \quad c_{k,k'} = a_1, \quad c_{j',k} = a_2. \]
In other words, our triangulation $\mathcal{T}$ has the shape as in Figure
\ref{fig_triangle}.
Using the Ptolemy relations (Theorem \ref{thm:ptolemy}) several times (as in the 
proof of Proposition \ref{prop:Delta1}) we obtain
$(c_{i,j}, c_{j,k}, c_{k,i}) = \Delta((a_1,a_2,b_1,b_2,c_1,c_2))$.
\end{proof}

Our goal is now to describe the image of $S_{\ge 0}=S\cap \ZZ_{\ge 0}^6$ under the map
$\Delta : S \rightarrow \ZZ^3$.
The strategy will be the following.
For a given triple $(a,b,c)\in\Delta(S_{\ge 0})\subseteq \mathbb{Z}_{\ge 0}^3$, 
imagine that the preimage of 
$(a,b,c)$ under $\Delta$ is an iceberg. We are looking for a 
non-negative preimage, 
i.e.\ for a place on the iceberg which is not under water. In a first step we find 
any place on the iceberg (Theorem \ref{triangles_Z}), the second step 
(Theorem \ref{exi}) is to move this solution (using Lemma \ref{eta_triangle}) to 
the dry peak of the iceberg.

Theorem \ref{val} clarifies the precise condition for $(a,b,c)$ to be in 
the image, Theorem \ref{mres} collects all these results and states the 
complete classification of triangles appearing in classic Conway-Coxeter friezes. 

\begin{Definition}
For a prime number $p$ and $n\in\NN$, we denote by $\nu_p(n)\in\ZZ_{\ge 0}$ the exponent of $p$ in the prime factorization of $n$, i.e.\
$p^{\nu_p(n)}$ divides $n$ but $p^{\nu_p(n)+1}$ does not divide $n$. We call $\nu_p(n)$ the \emph{$p$-valuation} of $n$.
\end{Definition}

\begin{Theorem}\label{triangles_Z}
Let $a,b,c\in \NN$ be such that
$d:=\gcd(a,b)=\gcd(b,c)=\gcd(a,c)$,
and assume that either
\[ \nu_2(a)=\nu_2(b)=\nu_2(c)=0\quad  \text{or} \quad
|\{\nu_2(a),\nu_2(b),\nu_2(c)\}|>1. \]
Then there exist $a_1,b_2\in\ZZ$ such that
\[ a_1 a + b b_2 = c \quad
\text{and}
\quad \gcd(a_1,b) = \gcd(a,b_2) = 1, \]
i.e.\ $\Delta((a_1,b,a-b_2,b_2,0,1)) = (a,b,c)$.
\end{Theorem}

\begin{Remark} \label{rem:valuation}
Suppose $a,b,c$ satisfy the assumptions in Theorem \ref{triangles_Z}. 
Then either all of $a,b,c$ are odd, or we have 
$|\{\nu_2(a),\nu_2(b),\nu_2(c)\}|=2$ and the maximal value of 
$\{\nu_2(a),\nu_2(b),\nu_2(c)\}$ is attained only once. (In fact, for the second
part use the condition on the greatest common divisors being equal.)
\end{Remark}

\begin{proof}
We set $a':=\frac{a}{d}$, $b':=\frac{b}{d}$ and $c':=\frac{c}{d}$.
In particular, $a'$ and $b'$ are coprime, so 
we may choose $u,v\in\ZZ$ such that
\begin{equation} \label{eq:bezout}
    u a'+v b' = 1.
\end{equation}    
Then multiplication by $c$ yields
\begin{equation}u c' a + v c' b = c. 
\end{equation}
For any $k\in\ZZ$ set
\[ a_1:=uc'+kb', \quad b_2:=v c'-ka'. \]
Then
\[ a_1 a + b b_2 = (uc'+kb') a + b(v c'-ka') = c \]
since $kb'a=bka'$ by definition of $a',b'$.
\smallskip

It remains to find some $k\in \ZZ$ such that in addition
$\gcd(a_1,b) = \gcd(a,b_2) = 1$.
\smallskip

We first claim that $\gcd(a_1,b')\mid c'$; in fact, we have
$$\gcd(a_1,b')=\gcd(uc'+kb',b')=\gcd(uc',b') \mid uc'$$
and since $b'$ and $u$ are coprime by (\ref{eq:bezout})
the claim follows.

Moreover, $b'$ and $c'$ are coprime by definition, so the above claim
yields $\gcd(a_1,b')=1$ and we obtain
\begin{equation}\label{cop1}
\gcd(a_1,b) = \gcd(a_1,b'd) = \gcd(a_1,d).
\end{equation}
Similarly,
\begin{equation}\label{cop2}
\gcd(a,b_2) = \gcd(b_2,d).
\end{equation}

This means that it suffices to find suitable $k\in \mathbb{Z}$ such that
$\gcd(a_1,d)=1=\gcd(b_2,d)$. 
\smallskip

If $d=1$ then we are clearly done. 
So assume $d>1$ and 
let $d=\prod_{i=1}^r p_i^{e_i}$ be the prime factorization (so 
$p_i\neq p_j$ for $i\neq j$). The idea now is to consider each prime divisor 
$p_i$ separately and find possible numbers $k_i$ such that 
$\gcd(a_1,p_i)=1=\gcd(b_2,p_i)$. This leads to different congruences and finally
an overall suitable $k\in \mathbb{Z}$ such that $\gcd(a_1,p_i)=1=\gcd(b_2,p_i)$
for all $i=1,\ldots, r$ is constructed by using the 
Chinese Remainder Theorem. 
\smallskip

Consider first an $i$ with $p_i>2$.
If $p_i$ divides $b'$, then $p_i$ does not divide $c'$ (by definition)
and $p_i$ does not divide $u$ by (\ref{eq:bezout}).
Thus we have
$\gcd(uc',p_i)=1$ and hence 
$$\gcd(a_1,p_i)=\gcd(uc'+kb',p_i)=1$$ 
for any $k$.
If $p_i$ does not divide $b'$, then $\gcd(a_1,p_i)=\gcd(uc'+kb',p_i)=1$ for all but
one $k\in \{0,\ldots,p_i-1\}$ (use that $b'$ is invertible modulo $p_i$).
Thus we get in all circumstances 
\[
|\{ k\in \{0,\ldots,p_i-1\} \mid \gcd(a_1,p_i)=1 \}| \ge p_i-1.
\]
It follows by symmetry that
\[
|\{ k\in \{0,\ldots,p_i-1\} \mid \gcd(p_i,b_2)=1 \}| \ge p_i-1.
\]
But $2(p_i-1)>p_i$ by assumption, thus we find a $k_i\in \{0,\ldots,p_i-1\}$ such that
$\gcd(a_1,p_i) = \gcd(p_i,b_2) = 1$.
\smallskip

If $p_i=2$, then $d$ and hence all of $a,b,c$ are even. It
follows from Remark \ref{rem:valuation} that precisely 
one of $a',b',c'$ is even.

If $c'$ is even, then $a',b'$ are odd. But then $a_1$ and $b_2$ are both congruent to $k$ modulo $2$, thus it suffices to choose any $k_i:=k$ odd.

The last case is when $a'$ or $b'$ is even. By symmetry we can assume that $a'$ is 
even (and hence $b'$ and $c'$ are odd).
But then $a_1\equiv u+k \:\:(\md 2)$ and $b_2\equiv v \equiv 1 \:\:(\md 2)$, where
the last congruence follows from (\ref{eq:bezout}).
Thus we may choose any $k_i:=k$ such that $u+k$ is odd.
\smallskip

Note that in all the above cases, the choice of $k_i$ is independent 
of adding multiples of $p_i$. So this yields a system of congruences
$$k \equiv k_i \:\:(\md p_i)\mbox{\hskip0.3cm for $i=1,\ldots,r$}.
$$
The Chinese Remainder Theorem now gives a solution for $k$ as desired.
\end{proof}

To understand the image of $\Delta$, we still need some more preparations.
The following lemma describes a special element in a nice group of 
transformations leaving $\Delta$ invariant.

\begin{Lemma}\label{eta_triangle}
Let $x:=(a_1,a_2,b_1,b_2,c_1,c_2)\in S$, $t\in\ZZ$ and
$$ \Gamma_t(x):=(a_1 t - a_2, a_1(1-t) + a_2, -b_2, b_1 + b_2(t + 1), c_1t + c_2(t - 1), c_1 + c_2).$$
Then $\Gamma_t(x)\in S$ and $\Delta(x)=\Delta(\Gamma_t(x))$.
\end{Lemma}
\begin{proof}
For the first claim note that $\Gamma_t(x)$ is given by a block diagonal matrix with $2\times 2$-matrices in $\SL_2(\mathbb{Z})$ on the diagonal.
Therefore elements in $S$ are mapped to elements in $S$.
For the second claim, just evaluate $\Delta$.
\end{proof}

\begin{Theorem}\label{exi}
Let $a,b,c\in \NN$ be such that
$d:=\gcd(a,b)=\gcd(b,c)=\gcd(a,c)$,
and assume that either
\[ \nu_2(a)=\nu_2(b)=\nu_2(c)=0\quad  \text{or} \quad
|\{\nu_2(a),\nu_2(b),\nu_2(c)\}|>1. \]
Then there exists a
tuple $(a_1,a_2,b_1,b_2,c_1,c_2)\in S\cap \ZZ_{\ge 0}^6$ such that
\[ \Delta((a_1,a_2,b_1,b_2,c_1,c_2)) = (a,b,c). \]
\end{Theorem}
\begin{proof}
By symmetry,
we may assume that $c\le a$ and $c\le b$.
Since we have the same assumptions as in Theorem \ref{triangles_Z}, we get
$a_1,b_2\in\ZZ$ such that
$a_1 a + b b_2 = c$
and
$\gcd(a_1,b) = \gcd(a,b_2) = 1$. At least one of $a_1,b_2$ must be non-negative
(since $a,b,c\in \mathbb{N}$),
so without loss of generality we
can assume $a_1\ge 0$.

If $a_1=0$ then $b=1$ since $\gcd(a_1,b)=1$;
but then
$\Delta((0,1,a-c,c,0,1))=(a,b,c)$ and we are finished (notice that $\gcd(a-c,c)=\gcd(a,c)=\gcd(a,b)=1$).

Thus assume that $a_1>0$ and let $a_2:=b$, $b_1:=a-b_2$, $c_1:=0$, and $c_2:=1$.
Then
$\Delta((a_1,a_2,b_1,b_2,c_1,c_2)) = (a,b,c)$ and indeed $(a_1,a_2,b_1,b_2,c_1,c_2)\in S$.
However, $b_1$ and $b_2$ may be negative in general, although $a_1,a_2,c_1,c_2\ge 0$.

We now show that applying the transformation $\Gamma_t$ of Lemma \ref{eta_triangle} several times eventually produces a tuple in $S\cap \ZZ_{\ge 0}^6$. In fact, let $t:=\lceil a_2/a_1 \rceil \in \mathbb{N}$ and consider
$$\Gamma_t((a_1,a_2,b_1,b_2,c_1,c_2))=:(\tilde a_1,\tilde a_2,\tilde b_1,\tilde b_2,\tilde c_1,\tilde c_2)\in S.$$
Then by definition of $t$, we have
\begin{eqnarray*}
\tilde a_1 &=& a_1 t - a_2 \ge 0, \\
\tilde a_2 &=& a_1(1-t) + a_2 > 0, \\
\tilde c_1 &=& c_1t + c_2(t - 1) = t-1 \ge 0, \\
\tilde c_2 &=& c_1 + c_2 = 1 > 0.
\end{eqnarray*}
Moreover, we claim that $\tilde b_1=-b_2\ge 0$; in fact,
$a_1a+bb_2 = c \le a$ and since $a_1a\ge a$ we conclude $b_2\le 0$. 

If $\tilde b_2\ge 0$ as well, then we are finished since
$(\tilde a_1,\tilde a_2,\tilde b_1,\tilde b_2,\tilde c_1,\tilde c_2)\in S\cap \ZZ_{\ge 0}^6$ as desired.
Otherwise, using that $b_2\le 0$ we get
\begin{eqnarray*}
0 & > & \tilde b_2 = b_1 + b_2(t + 1) = b_1 + b_2 (\lceil a_2/a_1 \rceil+1) \ge b_1 + b_2 (a_2/a_1+2) \\
&=& \frac{1}{a_1} (a_1b_1 + a_2b_2 + 2 a_1 b_2) = \frac{c+ a_1 b_2}{a_1} 
= \frac{c}{a_1} + b_2> b_2.
\end{eqnarray*}

If $\tilde a_1=0$, then $\tilde a_2=1$ (recall that $\gcd(\tilde a_1,\tilde a_2)=1$) and $c=\tilde b_2\ge 0$ (use the definition of $\Delta$ and Lemma \ref{eta_triangle}) and we are finished.
Thus assume that $\tilde a_1>0$.
But then we may replace $(a_1,a_2,b_1,b_2,c_1,c_2)$ by $(\tilde a_1,\tilde a_2,\tilde b_1,\tilde b_2,\tilde c_1,\tilde c_2)$ and repeat the same argument.
Since $0> \tilde{b_2} > b_2$, by induction we will eventually 
obtain $(\tilde a_1,\tilde a_2,\tilde b_1,\tilde b_2,\tilde c_1,\tilde c_2)\in S\cap \ZZ_{\ge 0}^6$ as desired.
\end{proof}

We now prove a converse to Theorem \ref{exi}.

\begin{Theorem}\label{val}
Let $a,b,c\in \NN$ be such that
$d:=\gcd(a,b)=\gcd(b,c)=\gcd(a,c)$,
and assume that
there exists $(a_1,a_2,b_1,b_2,c_1,c_2)\in S\cap \ZZ_{\ge 0}^6$ such that
\[ \Delta((a_1,a_2,b_1,b_2,c_1,c_2)) = (a,b,c). \]
Then either
\[ \nu_2(a)=\nu_2(b)=\nu_2(c)=0\quad  \text{or} \quad
|\{\nu_2(a),\nu_2(b),\nu_2(c)\}|>1. \]
\end{Theorem}
\begin{proof}
We prove the claim indirectly. Assume that $\nu_2(a)=\nu_2(b)=\nu_2(c)=\ell\ge 1$, thus
we may write $a=2^\ell \tilde a$, $b=2^\ell \tilde b$, $c=2^\ell \tilde c$ with $\tilde a, \tilde b,\tilde c$ odd.
Let $(a_1,a_2,b_1,b_2,c_1,c_2)\in S\cap \ZZ_{\ge 0}^6$ be such that
$\Delta((a_1,a_2,b_1,b_2,c_1,c_2)) = (a,b,c)$.
Then by Definition \ref{def:S} we have 
\begin{eqnarray}
\label{eq:1} a &=& b_1(c_1+c_2)+b_2c_2, \\
\label{eq:2} b &=& a_1c_1 + a_2(c_1+c_2),\\
\label{eq:3} c &=& a_1(b_1+b_2)+a_2 b_2.
\end{eqnarray}
We claim that this implies
\begin{eqnarray*}
c_1+c_2 \equiv b_2 \equiv a_1 \:\: (\md \: 2), \\
c_1 \equiv a_2 \equiv b_1+b_2 \:\: (\md \: 2).
\end{eqnarray*}
In fact, consider equations (\ref{eq:1})-(\ref{eq:3}). Since $a,b,c$ are even
by assumption, the two summands on the right hand side of each equation
are both even or both odd. Suppose that in (\ref{eq:1}) both summands
are even; the other case is dealt with similarly, the details are left to
the reader. We distinguish two cases.
If $b_1$ is even then $b_2$ is odd (by definition of $S$), hence
$b_1+b_2$ is odd. Since $b_2$ is odd our assumption on (\ref{eq:1}) 
implies that $c_2$ is
even. Then $c_1$ is odd (again by definition of $S$), thus $c_1+c_2$ is odd. 
But $a_1,a_2$ can not both be even (by definition of $S$), so (\ref{eq:2})
implies that $a_1$ and $a_2$ are both odd. This proves the claim on
the above congruences
in case $b_1$ is even. If $b_1$ is odd then our assumption on (\ref{eq:1}) 
implies that $c_1+c_2$ is even. By definition of $S$ we conclude
that $c_1$ and $c_2$ are both odd. Then the assumption on (\ref{eq:1})
yields that $b_2$ is even. Then $b_1$ is odd (by definition of $S$)
and $b_1+b_2$ is odd. Now (\ref{eq:2}) implies that $a_1$ is even.
So $a_2$ is odd (by definition of $S$) and this completes the proof of the claim. 
\smallskip

Hence, for the values of $a_1,a_2,b_1,b_2,c_1,c_2,c_1+c_2$ modulo $2$, only the following three cases are possible
(recall that $b_1,b_2$ cannot both be even by definition of $S$):
\begin{center}
\begin{tabular}{c |c |c |c |c |c |c} 
$b_1$ & $b_2$ & $a_1$ & $a_2$ & $c_1$ & $c_2$ & $c_1+c_2$ \\
\hline
0 & 1 & 1 & 1 & 1 & 0 & 1  \\
1 & 0 & 0 & 1 & 1 & 1 & 0  \\
1 & 1 & 1 & 0 & 0 & 1 & 1
\end{tabular}
\end{center}
By definition of $S$, the numbers $b_1$ and $b_2$ are coprime, so we can
choose $u,v\in\ZZ$ such that 
\begin{equation} \label{eq:bezout2}
ub_1-vb_2=1.
\end{equation}
Together with (\ref{eq:1})
this yields
$$b_1(c_1+c_2)+b_2c_2=a = aub_1-avb_2,
$$
hence $b_1\mid c_2+av$ and $b_2\mid c_1+c_2-au$ (use that $b_1$ and $b_2$
are coprime by definition of $S$).
So there exists $k_1,k_2\in \mathbb{Z}$ with 
\[ c_1+c_2 = au - k_1b_2, \quad c_2 = -av + k_2b_1. \]
Using Equations (\ref{eq:bezout2}), (\ref{eq:1}), $a= a+(k_1-k_2)b_1b_2$.
Hence we can choose $k_1=k_2$ (for the case $b_1=0$ or $b_2=0$, this is clear a priori).
More precisely, from the possible parities in the above table
one sees that $k$ has to be odd in all circumstances.

Now we get from equations (\ref{eq:2}) and (\ref{eq:3}) that
\begin{eqnarray*}
b &=& a_1 (a(u+v)-k(b_1+b_2)) + a_2 (au - kb_2), \\
c &=& a_1 (b_1+b_2)+a_2b_2
\end{eqnarray*}
which implies
\begin{eqnarray*}
b &=& a (a_1(u+v)+a_2 u) - k (a_1(b_1+b_2) + a_2 b_2) \\
 &=& 2^\ell \tilde a (a_1(u+v)+a_2 u) - k 2^\ell \tilde c \\
 &=& 2^\ell (\tilde a (a_1(u+v)+a_2 u) - k \tilde c).
\end{eqnarray*}
Thus because $\tilde a$, $\tilde c$ and $k$ are odd, and
$\tilde b = \tilde a (a_1(u+v)+a_2 u) - k \tilde c$ is odd as well,
we get that $a_1(u+v)+a_2 u$ is even. Using this fact, and the fact that 
$u,v$ are coprime by (\ref{eq:bezout2}), the above table extends to:
\begin{center}
\begin{tabular}{c | c | c |c |c |c }
$b_1$ & $b_2$ & $a_1$ & $a_2$ & $u$ & $v$ \\
\hline
0 & 1 & 1 & 1 & 1 & 0  \\
1 & 0 & 0 & 1 & 0 & 1  \\
1 & 1 & 1 & 0 & 1 & 1
\end{tabular}
\end{center}
But each of these three congruences for $b_1,b_2,u,v$ contradicts the equation 
(\ref{eq:bezout2}).
\end{proof}

As an immediate corollary to the Theorems \ref{exi} and \ref{val} we obtain the 
main result of this section, giving a complete classification of triples
appearing as labels of triangles in classic Conway-Coxeter friezes. 

\begin{Theorem}\label{mres}
Let $a,b,c\in \NN$ be such that
$d:=\gcd(a,b)=\gcd(b,c)=\gcd(a,c)$.
Then there exists $(a_1,a_2,b_1,b_2,c_1,c_2)\in S\cap \ZZ_{\ge 0}^6$ such that
\[ \Delta((a_1,a_2,b_1,b_2,c_1,c_2)) = (a,b,c) \]
if and only if either
\[ \nu_2(a)=\nu_2(b)=\nu_2(c)=0\quad  \text{or} \quad
|\{\nu_2(a),\nu_2(b),\nu_2(c)\}|>1. \]
\end{Theorem}

\begin{Theorem}\label{Delta}
Let $a,b,c\in \NN$.
Then the triple $(a,b,c)$ appears as labels of a triangle in a classic Conway-Coxeter frieze if and only if
$\gcd(a,b)=\gcd(b,c)=\gcd(a,c)$ and
\[ \nu_2(a)=\nu_2(b)=\nu_2(c)=0\quad  \text{or} \quad
|\{\nu_2(a),\nu_2(b),\nu_2(c)\}|>1. \]
\end{Theorem}
\begin{proof}
By Propositions \ref{prop:Delta1}, \ref{prop:Delta2}, and Lemma \ref{lem:gcd}, a triple
$(a,b,c)$ appears as labels of a triangle in a classic Conway-Coxeter frieze if and only if
$\gcd(a,b)=\gcd(b,c)=\gcd(a,c)$ and if
there exists a tuple $(a_1,a_2,b_1,b_2,c_1,c_2)\in S\cap \ZZ_{\ge 0}^6$ such that
\[ \Delta((a_1,a_2,b_1,b_2,c_1,c_2)) = (a,b,c). \]
The claim now follows by Theorem \ref{mres}.
\end{proof}

\section{Finiteness}
\label{sec:finiteness}

In this section we prove finiteness results on the number of possible 
frieze patterns with coefficients over subsets $R\subseteq\mathbb{C}$. 
Our main interest is in frieze patterns with coefficients over $\mathbb{N}$
but it will actually turn out that our finiteness result holds for 
arbitrary discrete subsets $R$ of $\mathbb{C}$. 

The following result is the key step. It is a generalization of 
an analogous result for classic frieze patterns, see \cite[Theorem 3.6]{CH}.

\begin{Lemma} \label{lem:finite}
Let $R\subseteq \mathbb{C}$ be a subset such that 
$$M:=\inf\{|x|\,:\,x\in R\setminus \{0\}\}>0.
$$
Let $\mathcal{C}=(c_{i,j})$ be a frieze pattern with coefficients of height 
$n\ge 1$
over $R\setminus \{0\}$.
Consider the boundary sequence $(c_{0,1},c_{1,2},\ldots,c_{n+2,n+3})$ and
set
$$P:=\max \{|c_{i,i+1}|\,:\,0\le i\le n+2\}.
$$
If $P\ge 1$ then every entry in the quiddity cycle of $\mathcal{C}$ has absolute value
at most 
$$\frac{P^2((PM+(n-1)P^2+M)}{M^2}.
$$
\end{Lemma}

\begin{proof}
We set $B:=\frac{P^2((PM+(n-1)P^2+M)}{M^2}$ for abbreviation. Note that
$B>0$ since $P>0$ and $M>0$ by assumption. 
To ease notation we also set $d_i:=c_{i,i+1}$ for all $i\in \mathbb{Z}$. 

Suppose for a contradiction that there exists an entry $x_1$ in the
quiddity cycle such that $|x_1|>B$.
Then we consider the corresponding two consecutive rows in the frieze pattern $\mathcal{C}$, with notation as in the following figure:
$$\begin{array}{cccccccccccc}
~0~ & ~d_j~ & ~x_1~ & ~x_2~ & ~\ldots~ & ~x_{n-1}~ & ~x_n~ & ~d_{j-1}~ & ~0~ & & & \\
& 0 & d_{j+1} & y_1 & y_2 & \ldots & y_{n-1} & y_n & d_{j} & ~0~ & & \\
\end{array}
$$
Note that the boundaries are as indicated because of the glide symmetry on 
the boundary entries, see Remark \ref{rem:frieze}.

The strategy of the proof is to proceed inductively along the rows.
More precisely, we want to use the following two conditions as 
induction hypotheses:
\begin{enumerate}
    \item \label{eq1} $|x_k| > \left( \frac{M(MB-P^2)}{P^2} - (k-2)P^2\right)\frac{1}{M}$
    for $k=2,\ldots, i+1$.
    \item \label{eq2} $|y_k| < \frac{MP^2}{M(MB-P^2)-(k-2)P^4} (P^2+M|x_{k+1}|)$
    for $k=2,\ldots, i$.
\end{enumerate}
For the base of the induction, let us check that condition (\ref{eq1})
holds for $k=2$ and $k=3$ and that condition (\ref{eq2}) holds for $k=2$.

By the defining condition for frieze patterns with coefficients, the 
triangle inequality, the definition of $P$ and our assumption that $P\ge 1$
we have
$$
|x_1y_1|  =  |d_jd_{j+2} + d_{j+1}x_2| 
\le  |d_jd_{j+2}| + |d_{j+1}x_2| \\
\le P^2 + P |x_2| \le P^2(1+|x_2|).
$$
Using the definition of $M$ and the assumption $|x_1|>B$ 
this yields
\begin{equation} \label{eq3}
    M\le |y_1| \le \frac{P^2(1+|x_2|)}{|x_1|} <\frac{P^2(1+|x_2|)}{B}.
\end{equation}
Solving for $|x_2|$ we obtain
\begin{equation} \label{eq4}
    |x_2| > \frac{MB}{P^2}-1 = \frac{MB-P^2}{P^2}
\end{equation}
which is condition (\ref{eq1}) for $k=2$. 

Going one step to the right, the defining condition for frieze
patterns with coefficients reads
$$x_2y_2-x_3y_1 = d_jd_{j+3}.
$$
By similar arguments as above and by using equations (\ref{eq3}) 
and (\ref{eq4}) we conclude
\begin{eqnarray*}
M & \le & |y_2| \le \frac{|d_jd_{j+3}| + |x_3y_1|}{|x_2|} 
\le P^2\,\frac{B+|x_3|+|x_3|\cdot |x_2|}{B|x_2|} \\
& = & P^2 \left( \frac{1}{|x_2|} + \frac{|x_3|}{B} \left(\frac{1}{|x_2|}+1\right)\right) \\
& < & P^2 \left( \frac{P^2}{MB-P^2} + \frac{|x_3|}{B}
\left(\frac{P^2}{MB-P^2}+1 \right)\right) \\
& = & \frac{P^2}{MB-P^2} \left( P^2 + M|x_3|\right).
\end{eqnarray*}
This shows equation (\ref{eq2}) for $k=2$.

Solving this inequality for $|x_3|$ then gives
$$|x_3| > \left( \frac{M(MB-P^2)}{P^2} -P^2\right) \frac{1}{M}
$$
and this is equation (\ref{eq1}) for $k=3$. 

Thus we have shown that we can indeed use conditions (\ref{eq1}) and
(\ref{eq2}) as induction hypotheses for an induction on $i$. 

For the induction step, we have to show that condition (\ref{eq2}) holds
for $i+1$ and that condition (\ref{eq1}) holds
for $i+2$. 

From the defining condition of frieze patterns with coefficients we have
$$x_{i+1}y_{i+1} - x_{i+2}y_i = d_j d_{j+i+2}
$$
and hence
$$|y_{i+1}| \le \frac{|d_jd_{j+i+2}| + |x_{i+2}|\cdot |y_i|}{|x_{i+1}|}
\le \frac{P^2+ |x_{i+2}|\cdot |y_i|}{|x_{i+1}|}
$$
Using the induction hypothesis for $|y_i|$ we deduce
$$|y_{i+1}| < \frac{P^2}{|x_{i+1}|} + 
\frac{MP^2 |x_{i+2}|}{M(MB-P^2)-(i-2)P^4} 
\left( \frac{P^2}{|x_{i+1}|}+M\right).
$$
Now we also plug in the induction hypothesis for $|x_{i+1}|$ and get
{\footnotesize
\begin{eqnarray*}
|y_{i+1}| & < & \frac{P^4M}{M(MB-P^2) - (i-1)P^4} + 
\frac{MP^2|x_{i+2}|}{M(MB-P^2)-(i-2)P^4}
\left( \frac{P^4M}{M(MB-P^2)-(i-1)P^4}+M \right) \\
& = & \frac{MP^2}{M(MB-P^2)-(i-1)P^4} \left( P^2+M|x_{i+2}|\right).
\end{eqnarray*}
}
This proves the induction step for $|y_{i+1}|$. 

Now we solve the last inequality for $|x_{i+2}|$. Since $M\le |y_{i+1}|$ 
this yields
$$
|x_{i+2}| > \left( \frac{M(MB-P^2)-(i-1)P^4}{P^2} - P^2\right)\frac{1}{M}\\
 =  \left( \frac{M(MB-P^2)}{P^2}-iP^2\right)\frac{1}{M}
$$
and this proves the inductive step for $|x_{i+2}|$. 
\smallskip

Noting that $|x_{n+1}|=|d_{j-1}|\le P$,
our induction argument eventually yields
\begin{equation} \label{eq5}
|y_n| < \frac{MP^3}{M(MB-P^2) - (n-2)P^4} (P+M).
\end{equation}
On the other hand, by definition of $B$ we have
$$M(MB-P^2)-(n-2)P^4 = P^3M + (n-1)P^4 -(n-2)P^4 = P^3(M+P).
$$
Together with equation (\ref{eq5}) this implies 
$|y_n| < M$, a contradiction to the definition of $M$ and the fact
that the frieze pattern $\mathcal{C}$ has non-zero entries by assumption.
This means that the assumption $|x_1|>B$ was wrong and hence every entry 
in the quiddity cycle of $\mathcal{C}$ has absolute value at most $B$,
as claimed.
\end{proof}

\begin{Remark}
Lemma \ref{lem:finite} does not hold without the assumption $P\ge 1$. As an example
consider the following frieze pattern with coefficients 
$$\begin{array}{cccccccc}
 & \ddots & & & & & & \\
0 & ~\frac{1}{2}~ & ~\frac{1}{2}~ & ~\frac{1}{2}~ & 0 & & & \\
& 0 & ~\frac{1}{2}~ & ~1~ & ~\frac{1}{2}~ & 0 & & \\
& & 0 & ~\frac{1}{2}~ & ~\frac{1}{2}~ & ~\frac{1}{2}~ & 0 & \\
& & & 0 & ~\frac{1}{2}~ & ~1~ & ~\frac{1}{2}~ & 0 \\
& & &  &  &  & \ddots & \\
\end{array}
$$
This is obtained from a classic Conway-Coxeter frieze pattern of height $n=1$ by scaling with the 
factor $\frac{1}{2}$, cf.\ Remark \ref{rem:scale}. With the notation of Lemma \ref{lem:finite}
we have $R=\frac{1}{2}\mathbb{N}$, the positive half-integers, and $M=\frac{1}{2}=P$.
Then the bound given in Lemma \ref{lem:finite} is equal to $\frac{3}{4}$ which is clearly not an 
upper bound for all entries in the quiddity cycle. 
\end{Remark}

As a consequence of Lemma \ref{lem:finite} we 
obtain the following finiteness result for frieze patterns with coefficients
over discrete subsets of complex numbers. This generalizes 
\cite[Corollary 3.8]{CH} from classic frieze patterns to frieze patterns with
coefficients. 

\begin{Proposition} \label{cor:finite}
Let $R\subseteq \mathbb{C}$ be a discrete subset, that is, $R$ has no accumulation
point. Let $n\in \mathbb{Z}_{\ge 0}$ and fix a boundary sequence in 
$(R\setminus \{0\})^{n+3}$. 
Then there are only finitely many frieze patterns with coefficients of height $n$
over $R\setminus \{0\}$ with the given boundary sequence. 
\end{Proposition}

\begin{proof}
For $n=0$ there is only one possible frieze pattern with coefficients for
any given boundary sequence, see Example \ref{ex:frieze}.

For $n\ge 1 $ we want to use Lemma \ref{lem:finite}.
Recall that frieze patterns with coefficients can be scaled (see Remark 
\ref{rem:scale}). The existence of only finitely many frieze patterns and the discreteness
of the set $R$ are invariant 
under scaling, so we can assume that our boundary sequence satisfies the assumption
$P\ge 1$ in Lemma \ref{lem:finite}. 

Since $R$ is discrete, the origin is not an accumulation point and $R$ 
satisfies the assumption of Lemma \ref{lem:finite}. Moreover, again by 
discreteness, every closed disk contains only finitely many elements from
$R$. Then Lemma \ref{lem:finite} implies that 
there are only finitely many elements of $R$ which can appear in the quiddity
cycle of a frieze pattern with coefficients having the given boundary 
sequence. However, any non-zero frieze pattern with coefficients is 
determined by the boundary sequence and the quiddity cycle, and this proves 
the claim.
\end{proof}

\end{document}